\documentclass{article}
\usepackage{a4wide}
\usepackage[utf8]{inputenc}
\usepackage{stmaryrd,mathrsfs,bm,amsthm,mathtools,yfonts,amssymb}
\usepackage{color}

\newtheorem{Theorem}{Theorem}[section]
\newtheorem{Proposition}{Proposition}[section]
\newtheorem{Lemma}{Lemma}[section]
\newtheorem{Corollary}{Corollary}[section]
\newtheorem{Definition}{Definition}[section]
\newtheorem{Remark}{Remark}[section]

\newcommand{\bTheorem}[1]{
\begin{Theorem} \label{T#1} }
\newcommand{\eT}{\end{Theorem}}

\newcommand{\bProposition}[1]{
\begin{Proposition} \label{P#1}}
\newcommand{\eP}{\end{Proposition}}

\newcommand{\bLemma}[1]{
\begin{Lemma} \label{L#1} }
\newcommand{\eL}{\end{Lemma}}

\newcommand{\bCorollary}[1]{
\begin{Corollary} \label{C#1} }
\newcommand{\eC}{\end{Corollary}}

\newcommand{\bFormula}[1]{
\begin{equation} \label{#1}}
\newcommand{\eF}{\end{equation}}

\newcommand{\pat}{\partial_t}
\newcommand{\diver}{\operatorname{div}}

\newcommand{\vu}{\vc{u}}
\newcommand{\vc}[1]{{\bf #1}}

\newcommand{\vq}{\vc{q}}

\newcommand{\vv}{\vc{v}}

\newcommand{\vw}{\vc{w}}

\newcommand{\TT}{\mathbb T}
\newcommand{\FF}{\mathfrak F}
\newcommand{\PP}{\mathbb P}

\newcommand{\vG}{\vc{G}}
\newcommand{\esssup}{{\rm esssup}}

\title{Stochastic forcing in hydrodynamic models with non-local interactions}
\author{Pavel Ludv\'ik$^\clubsuit$ \and V\'aclav M\'acha$^\spadesuit$}

\begin{document}
\maketitle
\bigskip

\centerline{$^{\clubsuit}$ Department of Mathematical Analysis and Applications of Mathematics,}
\centerline{Faculty of Science, Palack\'y University Olomouc}
\centerline{17. listopadu 12, 771 46 Olomouc, Czech Republic}
\centerline{email:pavel.ludvik@upol.cz}
\bigskip

\centerline{$^{\spadesuit}$ Institute of Mathematics of the Czech Academy of Sciences}
\centerline{\v Zitn\' a 25, 115 67 Praha 1, Czech Republic}
\centerline{email:macha@math.cas.cz}

\abstract{The hydrodynamical model of the collective behavior of animals consists of the Euler equation with additional non-local forcing terms representing the repulsive and attractive forces among individuals. This paper deals with the system endowed with an additional white-noise forcing and an artificial viscous term. We provide a proof of the existence of a dissipative martingale solution -- a cornerstone for a subsequent analysis of the system with stochastic forcing.}

\section{Introduction}
The modeling of the collective movement of animals is recently a popular topic. Its hydrodynamical model has a structure of  the Euler system equipped with additional forcing terms -- it can be obtained as a mean field limit of the Cucker-Smale flocking model \cite{CuSm}. Because of its hyperbolic nature, the aforementioned system is ill-posed even for smooth initial data. This was proven in the pioneering paper by DeLellis and Szekelehydi for the incompressible Euler system in \cite{DeSz} (check also \cite{ChKrMaSc} for smooth initial data and the compressible Euler system) and, later on, Carillo, Feireisl, Gwiazda and Swierczewska-Gwiazda extended the validity of this result also on the Euler system with non-local interaction in \cite{CaFeGwSw}.

In order to overcome the above-mentioned ill-possedness, we add an artificial viscous term. Therefore we deduce regularity properties allowing to conclude a proof of the existence of a solution. There is a hope that the elliptic  term does not play a significant role once the viscosity is assumed to be almost zero. The question of an inviscid limit will be addressed in a forthcoming paper. Here we would like to mention that the inviscid limit is often considered to be one of the criteria of admissibility for the Euler system. 

We tackle the Navier-Stokes system with extra terms representing the non-local interactions in a domain $(0,T)\times \mathbb T$ where $\mathbb T$ denotes a 3-dimensional torus, i.e., 
$
\TT = ([-1,1]|_{\{-1,1\}})^3
$. Namely,
\begin{equation}\label{main.sys}
\begin{split}
\pat \varrho + \diver(\varrho \vu)  = & 0,\\
\pat(\varrho \vu) +\diver (\varrho \vu \otimes \vu) + \nabla p(\varrho) - \diver S(D\vu) = &  -\varrho \int_{\TT}\nabla_x K(x-y)\varrho (y)\ {\rm d} y\\
&
+ \varrho \int_{\TT}\psi(x-y)\varrho(y)(\vu(y) - \vu(x))\ {\rm d}y.
\end{split}
\end{equation}
Here $\varrho:(0,T)\times \TT\mapsto [0,\infty)$ and $\vu:(0,T)\times \TT\mapsto\mathbb R^3$ are unknowns, namely a density and a velocity respectively. 

Further, $p(\varrho)$ is a given pressure satisfying
\begin{equation}
\begin{split}\label{p.law}
p\in C([0,\infty))\cap C^2(0,\infty),\ p(0)=0,\ p'(\varrho)>0\ \mbox{whenever}\ \varrho>0\\
\lim_{\varrho\to \infty} \frac{p'(\varrho)}{\varrho^{\gamma-1}} = p_\infty >0,\ \mbox{for some }\gamma>\frac32.
\end{split}
\end{equation}
We  introduce a potential $P$ by a formula
\begin{equation}\label{pres.potential}
P(\varrho) = \varrho\int_1^\varrho \frac{p(s)}{s^2}\ {\rm d}s.
\end{equation} 
Further, $S$ denotes a dissipative term and we assume that
\begin{equation}\label{viscosity.ass}
S(D\vu)= \mu(D\vu)  + \left(\lambda + \mu\right) \diver \vu \mathbb I,\ \mu>0
\end{equation}
where $D\vu$ is a symmetric part of the gradient of $\vu$ and $\mathbb I$ is the $3\times 3$ identity matrix and $\mu>0$, $\lambda - \frac23 \mu \geq 0.$

Next, $K = K(x)$ and $\psi = \psi(x)$ represent the non-local interaction forces acting on the medium. More specifically, the kernel $K$ includes the repulsive-attractive interaction force among individuals and $\psi$ gives the local averaging measuring the consensus in orientation. For more details we refer to \cite{CaFeGwSw}. These functions are assumed to satisfy
\begin{equation*}
\begin{split}
K\in C^2(\TT),\ \psi\in C^1(\TT),\ \psi\geq 0,\\
K\ \mbox{and}\ \psi\ \mbox{symmetric, i.e.,}\ K(x) = K(-x),\ \psi(x) = \psi(-x).
\end{split}
\end{equation*}

Taking $\mu, \lambda=0$, the system describes a collective behavior of animals -- dynamics of a school of fishes or a flock of birds -- the interested reader may find more in \cite{CaCaRo} and \cite{KaMeTr} where is a rigorous derivation of the Cucker-Smale flocking model. It is also worthwhile to mention that taking $K=0$, $\psi=0$, we obtain the compressible Navier-Stokes equation with pressure which is more general then the one considered in \cite{BrFeHo}.

The existence of a solution to such model has been already studied (even with an inviscid limit) in \cite{BrMa}.

Our aim is to look into the model equipped with an additional stochastic forcing term. Namely, we consider the following version of \eqref{main.sys}:
\begin{equation}\label{NS.stochastic}
\begin{split}
{\rm d}\varrho + \diver(\varrho\vu)\ {\rm d}t &= 0\\
{\rm d}\varrho \vu +\diver (\varrho \vu \otimes \vu)\ {\rm d}t + \nabla p(\varrho)\ {\rm d}t - \diver S(D\vu)\ {\rm d}t = &  -\varrho \int_{\TT}\nabla_x K(x-y)\varrho (y)\ {\rm d} y \ {\rm d}t\\
&
+ \varrho \int_{\TT}\psi(x-y)\varrho(y)(\vu(y) - \vu(x))\ {\rm d}y\ \ {\rm d}t\\
& +\mathbb G(\varrho,\varrho\vu)\ {\rm d}W \mbox{ in }(0,T)\times \TT.
\end{split}
\end{equation}
We consider unknowns to be dependent also on $\omega$ from a stochastic basis $(\Omega, \FF, (\FF_t)_{t\geq0})$ which allows us to interpret them as random variables or stochastic processes (or random distributions in a general setting).
Further, it is assumed that 
\begin{equation}\label{G.law.1}
\mathbb G(\varrho,\varrho\vu)\ {\rm d}W = \sum_{k=1}^\infty \vG_k(x,\varrho,\varrho\vu) \ {\rm d}W_k,
\end{equation}
where $W = (W_k)_{k\in\mathbb{N}}$ is a cylindrical $(\FF_t)$-Wiener process in a separable Hilbert space $\mathfrak U$ given by a formal sum $W(t) = \sum_{k=1}^{\infty} e_k W_k(t)$, where $(e_k)_{k\in\mathbb{N}}$ is a complete orthonormal system in $\mathfrak U$, $(W_k)_{k\in\mathbb{N}}$ is a sequence of mutually independent real-valued $(\FF_t)$-Wiener processes, and $\mathbb G(\varrho,\varrho\vu):\mathfrak U\mapsto L^1(\mathbb T)$. Further, the coefficients $\mathbf G_k=\mathbf G_k(x,\varrho,\vu) = \mathbb G(x,\varrho,\varrho\vu) e_k$ are continuously differentiable and we assume the existence of a sequence $\{g_k\}_{k=1}^\infty\subset [0,\infty)$ with $\sum_{k=1}^\infty g_k^2<\infty$ such that
\begin{equation}\label{G.law.2}
|\vG_k(x,\varrho,\vq)|\leq g_k(\varrho + |\vq|),\quad |\nabla_{\varrho,\vq} \vG_k(x,\varrho,\vq)|\leq g_k.
\end{equation}
We deal with the existence of a dissipative martingale solution (see Definition \ref{def.m.sol}). The paper is organized as follows. The rest of this chapter contains basic definitions, notions, and also the main theorem. The approximate system is introduced in Section \ref{approx.sys}, which also contains a proof of the existence of solution which is done by means of the Galerkin approximations. Section \ref{sec.vat} is devoted to the appropriate limits of the approximations. 

\subsection{Preliminaries}

For $p\geq 1$, the Lebesgue spaces of function $f:X\mapsto Y$ are denoted by $L^p(X,Y)$ and the Sobolev spaces of $k$-times differentiable functions by $W^{k,p}(X,Y)$. The dual space to $W^{k,p}(X,Y)$ is denoted by $W^{-k,p'}(X,Y)$ where $p' = \frac{p}{p-1}$. We will omit the set $Y$ in case we avoid any misunderstandings. We would like to mention that vector-valued quantities are denoted by bold letters whereas scalar-valued quantities are denoted by usual (non-bold) letters.

Let $\alpha\in (0,1)$. The spaces of $\alpha$-H\"older continuous functions $f:X\to Y$ are denoted by $C^{\alpha}(X,Y)$. Furthermore we use a natural convention that if $\alpha \in \mathbb N$ then $C^{\alpha}(X,Y)$ denotes the space of $\alpha$-times differentiable functions and for $\alpha \in (1,\infty)\setminus \mathbb N$,  $C^{\alpha}(X,Y)$ denotes functions whose derivative of order $\lfloor \alpha\rfloor$ is $(\alpha - \lfloor \alpha \rfloor)$-H\"older continuous. 
The space of functions $f:X\to Y$ whose derivative of arbitrary order is bounded is denoted by $C^\infty(X,Y)$. Spaces $C^\alpha_c (X,Y)$ are spaces of compactly supported functions belonging to $C^\alpha(X,Y)$. 

Before stating a definition of a solution let recall Ito's formula which is of extensive use.

\begin{Lemma}[consequence of {\cite[Theorem~4.17]{DPZ92}}]\label{thm:ito}
Let $X$ be a separable Hilbert space, $U$ and $V$ be $X$-valued stochastic processes, let $W = \sum_{k=1}^\infty e_k W_k$ be a cylindrical Wiener process in a separable Hilbert space $\mathfrak{U}$, let $G,\ H$ be  $L_2(\mathfrak U,X)$-valued stochastically integrable processes and let $g,h:[0,T]\mapsto X$ be progressively measurable Bochner integrable processes such that
\begin{equation*}
\begin{split}
U(t) & = U(0) + \int_0^t g(s)\ {\rm d}s + \int_0^t G(s) {\rm d}W(s)\\
V(t) & = V(0) + \int_0^th(s)\ {\rm d}s + \int_0^t H(s) {\rm d}W(s).
\end{split}
\end{equation*}
Then,
\begin{multline*}
\langle U(t),V(t)\rangle = \langle U(0),V(0)\rangle + \int_0^t \langle V(s), g(s)\rangle \ {\rm d}s  + \int_0^t \langle V(s), G(s) {\rm d}W(s)\rangle \\ +\int_0^t \langle U(s), h(s)\rangle \ {\rm d}s + \int_0^t \langle U(s), H(s)\ {\rm d}W(s)\rangle + \sum_{k=1}^\infty\int_0^t \langle G(s) e_k, H(s) e_k\rangle \ {\rm d}s
\end{multline*}
and
\begin{multline*}
F(U(t)) = F(U(0)) + \int_0^t  \langle F'(U(s)), g(s)\rangle \ {\rm d}s + \int_0^t \langle F'(U(s)), G(s)\ {\rm d}W(s)\rangle\\  + \frac 12 \int_0^t {\rm Tr} (G(s)^* F''(U(s)) G(s))\ {\rm d}s
\end{multline*}
whenever $F:X\mapsto \mathbb R$ is an arbitrary function whose derivatives $F'$ and $F'' $ are uniformly continuous on bounded subsets of $X$. 
\end{Lemma}

Another important tool will allow us to identify the stochastic limit integral. Before we do, 
let us define the space $\mathfrak{U}_0\supset\mathfrak{U}$ which is a normed space generated by $\mathfrak{U}$, considered with a complete orthonormal system $\{e_k\}_{k=1}^{\infty}$, as $\mathfrak{U}_0 := \left\{ v = \sum_{k=1}^{\infty} \alpha_k e_k; \sum_{k=1}^{\infty} \frac{\alpha_k^2}{k^2}<\infty\right\}$ with a norm $\|v\|_{\mathfrak{U}_0} := \sum_{k=1}^{\infty} \frac{\alpha_k^2}{k^2}$, where $v = \sum_{k=1}^{\infty} \alpha_k e_k$. The reason of considering this auxiliary space is to give a proper meaning to a sum in a definition of a cylindrical Wiener process $W(t) = \sum_{k=1}^{\infty} e_k W_k$.

\begin{Lemma}[{\cite[Lemma 2.6.6]{BrFeHo}}; a consequence of {\cite[Lemma 2.1]{DGHT11}}]\label{L266}
Let $(\Omega,\FF,\mathbb{P})$ be a complete probability space. For $n\in\mathbb{N}$, let $W_n$ be an $(\FF_t^n)$-cylindrical Wiener process and let $\mathbf{G}_n$ be an $(\FF_t^n)$-progressively measurable stochastic process such that $\mathbf{G}_n\in L^2([0,T],L_2(\mathfrak{U},W^{l,2}(\TT))$ a.s. Suppose that
\begin{equation*}\begin{split}
W_n &\to W \quad \mbox{in $C([0,T],\mathfrak{U}_0)$ in probability,}\\
\mathbf{G}_n &\to \mathbf{G} \quad \mbox{in $L_2(\mathfrak{U},W^{l,2}(\TT))$ in probability,}
\end{split}
\end{equation*}
where $W = \sum_{k=1}^{\infty} e_k W_k$. Let $(\FF_t)_{t\in\mathbb{N}}$ be the filtration given by
\[
\FF_t = \sigma \left(\bigcup_{k=1}^{\infty} \sigma_t[\mathbf{G}e_k] \cup \sigma_t[W_k]\right).
\]
Then after a possible change on a set of zero measure in $\Omega\times(0,T)$, $\mathbf{G}$ is $(\FF_t)$-progressively measurable, and
\[
\int_0^{\cdot} \mathbf{G}_n {\rm d} W_n \to \int_0^{\cdot} \mathbf{G}{\rm d}W \quad \mbox{ in $L_2(\mathfrak{U},W^{l,2}(\TT))$ in probability.}
\]

\end{Lemma}

The following {\em dissipative martingale solution} is a stochastic equivalent of a usual notion of weak solution supplemented with a certain energy balance condition. The concept was introduced in \cite{BFH17}.

\begin{Definition} \label{def.m.sol}
Let $\Lambda = \Lambda(\varrho,\vq)$ be a Borel probability measure on $L^1(\TT)\times L^1(\TT)$ such that
$$
\Lambda\{\varrho\geq 0\} = 1,\ \int_{L^1\times L^1}\left|\int_{\TT} \frac{|\vq|^2}\varrho + P(\varrho)\ {\rm d}x\right|^r\ {\rm d}\Lambda(\varrho,\vq) <\infty
$$
where $r\geq 1$ will be determined below.

The quantity $((\Omega,\FF, (\FF_t)_{t\geq0}, \PP), \varrho,\vu, W)$ is called a dissipative martingale solution to \eqref{NS.stochastic} with the initial law $\Lambda$ if
\begin{enumerate}
\item $(\Omega, \FF, (\FF_t)_{t\geq0},\PP)$ is a stochastic basis with a complete right-continuous filtration,
\item $W = (W_k)_{k\in\mathbb{N}}$ is a cylindrical $(\FF_t)$-Wiener process,
\item the density $\varrho$ and the velocity $\vu$ are random distributions adapted to $(\FF_t)_{t\geq0}$, $\varrho\geq0$ $\PP$-a.s.,
\item there exists an $\FF_0$-measurable random variable $[\varrho_0,\vu_0]$ such that $\Lambda = \mathcal L[\varrho_0,\varrho_0\vu_0]$,
\item the equation of continuity
\begin{equation*}
\int_0^T\int_\TT \varrho \varphi \pat \Phi \ {\rm d}x{\rm d}t + \int_{\TT} \varrho_0\varphi\ {\rm d}x \Phi(0) + \int_0^T\int_\TT \varrho \vu \nabla \varphi \Phi\ {\rm d}x{\rm d}t = 0
\end{equation*}
holds for all $\Phi\in C^\infty_c([0,T))$ and all $\varphi \in C^\infty(\TT)$ $\PP$-a.s.,
\item the momentum equation 
\begin{multline*}
\int_0^T\int_\TT \varrho \vu \cdot \varphi \pat \Phi\ {\rm d}x{\rm d}t  + \int_\TT \varrho_0\vu_0 \varphi \Phi(0)\ {\rm d}x \\
+ \int_0^T \int_\TT \varrho \vu\otimes\vu:\nabla \varphi + p(\varrho)\diver \varphi \Phi\ {\rm d}x{\rm d}t\\
 - \int_0^T\int_\TT S(D\vu): D\varphi \Phi \ {\rm d}x{\rm d}t = \\
\int_0^T\int_\TT \varrho \nabla K*\varrho  \varphi \Phi\ {\rm d}x{\rm d}t - \int_0^T\int_\TT\varrho(t,x) \varphi(x) \int_\TT \varrho(t,y) \psi(x-y) (\vu(t,y) - \vu(t,x)) \ {\rm d}x{\rm d}y{\rm d}t\\
+ \sum_{k=1}^\infty\int_0^T\int_\TT \vG_k(\varrho,\varrho\vu)\cdot\varphi \Phi\ {\rm d}x{\rm d}W_k
\end{multline*}
holds for all $\Phi\in C_c^\infty([0,T))$ and all $\varphi \in C^\infty(\TT)$ $\PP$-a.s.,
\item the energy inequality
\begin{multline} \label{eq:ene.ine}
\int_\TT \left(\frac 12 \varrho(T,x) |\vu(T,x)|^2 + P(\varrho(T,x)) + \frac12 \int_\TT K(x-y)\varrho(T,y)\varrho(T,x)\ {\rm d}y \right)\Phi(T)\ {\rm d}x\\
 - \int_\TT\left(\frac 12 \varrho_0(x) |\vu_0(x)|^2 + P(\varrho_0(x)) + \frac12 \int_\TT K(x-y)\varrho_0(y)\varrho_0(x)\ {\rm d}y \right)\Phi(0)\ {\rm d}x\\
 + \int_0^T\int_\TT S(D\vu):D\vu \Phi(t)\ {\rm d}x{\rm d}t\\
 + \frac12\int_0^T\int_\TT\int_\TT \varrho(t,x)\varrho(t,y)\left(\vu(t,y)-\vu(t,x)\right)^2 \psi(x-y)\Phi(t)\ {\rm d}x{\rm d}y{\rm d}t\\
\leq \frac 12\int_0^T \int_\TT \sum_{k=1}^\infty \varrho^{-1}|\vG_k(\varrho,\varrho\vu)|^2 \Phi(t)\ {\rm d}x{\rm d}t + \sum_{k=1}^\infty \int_0^T\int_\TT \vG_k(\varrho,\varrho\vu)\vu\Phi\ {\rm d}x{\rm d}W_k
\end{multline}
holds for all $\Phi\in C^\infty([0,T])$, $\Phi\geq 0$.
\end{enumerate}
\end{Definition}

Note that \eqref{eq:ene.ine} may be deduced from \eqref{NS.stochastic} and the Ito's formula under additional assumption that all quantities are sufficiently regular. Indeed,
assume $\varrho\geq \underline\varrho >0$. We deduce from \eqref{NS.stochastic}$_1$ that
\begin{equation*}
\pat \frac 1\varrho  = \frac{\diver(\varrho \vu)}{\varrho^2}.
\end{equation*}
We use Ito's formula (Theorem \ref{thm:ito}) for $\frac 1\varrho (\varrho\vu)(\varrho \vu)\Phi$ where $\Phi\in C^\infty([0,T]),\ \Phi\geq 0$. As a matter of fact \eqref{eq:ene.ine} holds with the equality sign once the solution is smooth enough.

\subsection{Main result}

Our main result deals with the existence of the aforementioned notion of solution:

\begin{Theorem}
Let $\Lambda$ be a Borel probability measure defined on the space $L^1(\TT)\times L^1(\TT)$ such that 
$$
\Lambda\{\varrho\geq 0\} = 1,\quad \Lambda\left\{0<\underline\varrho \leq \int_\TT \varrho\ {\rm d}x\leq \overline \varrho <\infty\right\} = 1
$$
for some $\underline\varrho$ and $\overline \varrho$ and 
$$
\int_{L^1\times L^1} \left|\int_\TT \left[\frac12\frac{|{\bf q}|^2}{\varrho} + P(\varrho)\right]\ {\rm d}x\right|^r\ {\rm d}\Lambda \leq c<\infty
$$
for some $r\geq 4$. Let \eqref{p.law}, \eqref{G.law.1} and \eqref{G.law.2} be assumed.

Then there is a dissipative martingale solution to \eqref{NS.stochastic}
\end{Theorem}
The rest of this paper is devoted to the proof of this theorem.

The strategy is as follows. First, we introduce an approximate system. There are basically two types of approximation. First, we add the artificial viscosity term $\varepsilon \Delta \varrho$ into the continuity equation -- it is important to add also $\varepsilon\Delta(\varrho \vu)$ to the momentum equation to keep the energy inequality valid. Secondly, we consider a pressure $p_\delta(\varrho) = p(\varrho) + \delta(\varrho^2 + \varrho^6)$. Both these approximations increase the regularity of the density $\varrho$ which allows to use the method of the Galerkin approximations. 

\section{Approximate system}
\label{approx.sys}
We consider the following system
\begin{equation}
\begin{split}\label{NS.approx1}
{\rm d}\varrho + \diver(\varrho\vu)\ {\rm d}t =&\ \varepsilon \Delta\varrho \ {\rm d}t\\
{\rm d}\varrho \vu +\diver (\varrho \vu \otimes \vu)\ {\rm d}t + \nabla p_\delta(\varrho)\ {\rm d}t - \diver S(D\vu)\ {\rm d}t = &\ \varepsilon \Delta(\varrho u)\ {\rm d}t  -\varrho \int_{\TT}\nabla_x K(x-y)\varrho (y)\ {\rm d} y \ {\rm d}t\\
&
+ \varrho \int_{\TT}\psi(x-y)\varrho(y)(\vu(y) - \vu(x))\ {\rm d}y\ {\rm d}t\\
& +\mathbb G_{\varepsilon}(\varrho,\varrho\vu)\ {\rm d}W \mbox{ in }(0,T)\times \TT.
\end{split}
\end{equation}
where $p_\delta(\varrho) = p(\varrho)  +\delta (\varrho^\Gamma + \varrho^2)$ (here $\Gamma = \max\{\gamma,6\}$) and we define a potential $P_\delta$ in agreement with \eqref{pres.potential}, i.e.,
\begin{equation}\label{def.pdelta}
P_\delta(\varrho) = \varrho \int_1^\varrho \frac{p_\delta(s)}{s^2}\ {\rm d}s.
\end{equation}
$\mathbb G_\varepsilon$ is defined in the next subsection (see \eqref{eq:Gedef}). Moreover, we assume that
\begin{equation}\label{ini.rho.smooth}
\varrho(0,\cdot) = \varrho_0(\cdot),\quad \varrho_0\in C^{2+\nu}(\TT),\ 0<\underline\varrho \leq \varrho_0\leq \overline\varrho<\infty\quad \mbox{almost surely}
\end{equation}
for some $\underline\varrho\in \mathbb R$ and there is $M>0$ such that
$$
\int_\TT\varrho_0\ {\rm d}x = M
$$ 
for all $\omega\in \Omega$.

The additional term $\varepsilon\Delta \varrho$ changes the character of the continuity equation -- it becomes of parabolic type which yields sufficient regularity of solution. Next, the additional term $\delta(\varrho^\Gamma + \varrho^2)$ in the pressure $p_\delta$ increases the regularity of $\varrho$ which then allows to pass with $\varepsilon$ to $0$. In what follows, we first show the existence of solution to \eqref{NS.approx1} by time-discretization and the Galerkin approximations. Then we pass with $\varepsilon$ to $0$ and, subsequently, we tend with $\delta$ to zero.

\subsection{Time-discretization and finite-dimensional space}
Let $X_m$ denotes 
an $m$-dimensional subspace of $L^2(\TT)$ spanned by eigenvectors of the Laplace operator, i.e., $\vw_n\in L^2(\TT)$ is a sequence of functions fulfilling
\begin{equation*}
-\Delta \vw_n = \lambda_n \vw_n
\end{equation*}
for some $\{\lambda_n\}_{n=1}^\infty\subset \mathbb R$, such that
\begin{equation*}
L^2(\TT) = \overline{{\rm span}\{\vw_1, \vw_2, \ldots\}}^{L^2}.
\end{equation*}
We define
\begin{equation*}
X_m = {\rm span}\{\vw_1, \vw_2, \ldots, \vw_m\}. 	
\end{equation*}
Further, $\Pi_m:L^2(\mathbb T)\to X_m$ denotes an orthogonal projection.

Take a non-increasing smooth cut-off function fulfilling
\begin{equation*}
\chi(z) = \left\{\begin{array}{ll}
1 & \ \mbox{for }z\leq 0,\\
0 & \ \mbox{for }z\geq 1.
\end{array}\right.
\end{equation*}
We define for any $\vv\in X_m$ and $R\in \mathbb R$
\begin{equation*}
[\vv]_R = \chi(\|\vv\|_{X_m} - R)\vv.
\end{equation*}

For a given stochastic basis $(\Omega, \FF, (\FF_t)_{t\geq 0}, \mathbb P)$ with a complete right-continuous filtration and a cylindrical $(\FF_t)$-Wiener process $W$ we start with an approximation problem
\begin{equation}
\begin{split}\label{eq:ap.prob}
{\rm d}\varrho + \diver(\varrho[\vu]_R){\rm d}t & = \varepsilon \Delta\varrho {\rm d}t\\
{\rm d}\Pi_m(\varrho\vu) + \Pi_m(\diver(\varrho[\vu]_R\otimes\vu)) {\rm d}t + \Pi_m (\chi(\|\vu\|_{X_m} - R)\nabla p_\delta(\varrho)) {\rm d}t & = \Pi_m(\varepsilon \Delta(\varrho \vu)){\rm d}t\\
	+ \Pi_m(\diver S(D\vu)) {\rm d}t& - \Pi_m(\varrho(\nabla K*\varrho))\ {\rm d}t\\
	+ \Pi_m(\varrho (\psi*(\varrho\vu))& - \varrho\vu (\psi*\varrho))\ {\rm d}t\\
	+ \Pi_m(\varrho \Pi_m(\mathbb F_\varepsilon(\varrho,\vu)))\ {\rm d}W.
\end{split}
\end{equation}

Here, the right hand side $\mathbb F_\varepsilon$ is defined as
\begin{equation*}
\mathbb F_{\varepsilon}(\varrho,\vu) = \sum_{k=1}^\infty F_{k,\varepsilon}(\varrho,\vu) e_k
\end{equation*}
where
\begin{equation*}
F_{k,\varepsilon}(\varrho,\vu) = \chi\left(\frac \varepsilon\varrho - 1\right)\chi\left(|\vu| - \frac 1\varepsilon\right) \frac 1\varrho {\bf G}_k(\varrho,\varrho\vu).
\end{equation*}
Due to \eqref{G.law.2} there exists a sequence $\{f_{k,\varepsilon}\}_{k=1}^\infty\subset \mathbb R$ such that
\begin{equation}\label{eq:bound.F}
\|F_{k,\varepsilon}\|_{L^\infty} + \|\nabla_{\varrho,\vu} F_{k,\varepsilon}\|_{L^\infty} \leq f_{k,\varepsilon},\quad \sum_{k=1}^\infty f_{k,\varepsilon}^2 <\infty.
\end{equation} 
Moreover, there exists a sequence of numbers $\{f_k\}$ independent of $\varepsilon$ such that
\begin{equation}
\|F_{k,\varepsilon} (\cdot,\cdot,0)\|_{L^\infty(\TT\times\mathbb R)} + \|\nabla_\vu F_{k,\varepsilon}\|_{L^\infty(\TT\times\mathbb R\times\mathbb R^3)}\leq f_k,\ \sum_{k=1}^\infty f_k^2<\infty.\label{eq:bound.F1}
\end{equation}
This follows directly from \eqref{G.law.2} by a chain rule.
We also define $\mathbb G_\varepsilon$ by a formula
\begin{equation}\label{eq:Gedef}
\mathbb G_\varepsilon(\varrho,\vu) = \sum_{k=1}^\infty G_{k,\varepsilon}(\varrho,\varrho\vu)e_k
\end{equation}
where 
\begin{equation*}
G_{k,\varepsilon}(\varrho,\varrho\vu) = \varrho F_{k,\varepsilon}\left(\varrho,\frac{\varrho\vu}{\varrho}\right)
\end{equation*}

Let us note that since $\varrho,\ \vu$ are $(\FF_t)$-adapted by definition and have continuous trajectories $\PP$-a.s., they are also $(\FF_t)$-progressively measurable (it follows from the proof of \cite[Proposition 1.13]{KS91}). Utilizing \eqref{eq:bound.F} and properties of the projection $\Pi_m$, the same holds true for the composition $\varrho \Pi_m(\mathbb F_\varepsilon(\varrho,\vu))$ and moreover, it is a mapping with a range in $L_2(\mathfrak U, W^{-b,2}(\TT))$ for $b>\frac32$ and the stochastic integral in \eqref{eq:ap.prob} is well-defined (for details see \cite[Remark 4.1.2]{BrFeHo}).

We take a time step $h>0$ and set 
\begin{equation*}
\varrho(t) = \varrho_0,\qquad \vu(t) = \frac{(\varrho \vu)_0}{\varrho_0} \quad \mbox{ for $t\leq0$}
\end{equation*}
where we assume that $\mathbb E[\|\vu(0)\|_{X_m}^r]$ is bounded independently of $m$.
We define a solution $\varrho$, $\vu$ as follows: in each time step $t\in [nh,nh+h)$ let $\varrho(t)$ be a solution to
\begin{equation}
\begin{split}\label{eq:ap.prob2}
\pat\varrho +  \diver (\varrho [\vu(nh)]_R) & = \varepsilon \Delta\varrho\\
\varrho(nh) & = \lim_{t\to nh_-} \varrho(t).
\end{split}
\end{equation}
Note that this is a usual parabolic problem -- for  more we refer for example to \cite{DeHiPr}. Here we also infer the mass conservation law, i.e.
\begin{equation}\label{eq:mass.cons}
\int_\TT \varrho(t,\cdot)\ {\rm d}x = \int_\TT \varrho_0 (\cdot)\ {\rm d}x = M
\end{equation}
which holds for all $\omega\in \Omega$ and all $t\in (0,T)$. This follows by integrating \eqref{eq:ap.prob2} over $(0,t)\times\TT$. 

Further, we use a nowadays usual division by $\varrho$ represented by a functional $\mathcal M^{-1}_\varrho$ defined as an inverse to a functional
\begin{equation*}
\mathcal M_\varrho:X_m\mapsto X_m,\ \int_{\mathbb T}\mathcal M_\varrho \vv\cdot \varphi\ {\rm d}x = \int_{\mathbb T} \varrho \vv \cdot \varphi \ {\rm d}x \ \forall \varphi \in X_m
\end{equation*}
(for more properties of $\mathcal M^{-1}$ see \cite[Section 2.2]{FeNoPe}) to define $\vu$ in $[nh, nh+h)$ as 
\begin{multline}\label{eq:ap.prob3}
\vu(t) = \mathcal M^{-1}_\varrho (\Pi_m(\varrho \vu(nh))) - \mathcal M^{-1}_\varrho\int_{nh}^t \Pi_m (\diver(\varrho(s)[\vu(nh)]_R\otimes \vu(nh)))\ {\rm d}s\\
 - \mathcal M^{-1}_{\varrho} \int_{nh}^t \Pi_m (\chi(\|\vu(nh)\|_{X_m} - R) \nabla p_{\delta}(\varrho(s)))\ {\rm d}s \\
+ \mathcal M^{-1}_{\varrho} \int_{nh}^t \Pi_m (\varepsilon\Delta(\varrho (s)\vu(nh)) + \diver S(D\vu(nh)))\ {\rm d}s\\
 + \mathcal M^{-1}_\varrho \int_{nh}^t \Pi_m (\varrho(\psi * (\varrho\vu(nh))) - \varrho \vu(nh) (\psi*\varrho) - \varrho (\nabla K*\varrho))\ {\rm d}s\\
 + \mathcal M^{-1}_\varrho \int_{nh}^t \Pi_m (\varrho(s) \Pi_m (\mathbb F_\varepsilon (\varrho(nh),\vu(nh))))\ {\rm d}W
\end{multline}

Due to the known results concerning the regularity of a solution to the heat equation, we obtain functions $\varrho,\ \vu$ satisfying
\begin{equation}\label{eq:h.con.rho}
\varrho \in C([0,T]; C^{2+\nu}(\mathbb T)),\quad \varrho>0,\qquad \vu \in C([0,T]; X_m)\qquad \mathbb P\mbox{-a.s.}
\end{equation}
for some $\nu>0$, where $\varrho,\ \vu$ are progressively $(\FF_t)$-measurable. 

Both $\varrho$ and $\vu$ have continuous trajectories $\mathbb P$-a.s. and $\varrho$ is $(\FF_t)$-adapted (thanks to \eqref{eq:ap.prob2} and initial conditions) and $\vu$ is $(\FF_t)$-adapted since the stochastic integral in \eqref{eq:ap.prob3} is well-defined (similar arguments as below \eqref{eq:ap.prob}) and as such is a $(\FF_t)$-martingale.

By a standard interpolation -- see e.g. \cite[Proposition 1.1.3]{lunardi} -- one may deduce that 
\begin{equation}\label{eq:rho.reg}
\varrho\in C^\nu([0,T];C^{2+\nu}(\mathbb T)) \qquad \mathbb P\mbox{-a.s.}
\end{equation}
for some $\nu>0$ which is possibly smaller than the one mentioned in \eqref{eq:h.con.rho}.

\subsection{Time smoothening}\label{S:2.2}
Let $\varrho_h$ and $\vu_h$  be the solution constructed in the previous section corresponding to a certain value of $h>0$. In order to pass with $h$ to $0$ we deduce estimates independent of $h$. 
Assume that 
\begin{equation*}
0<\underline\varrho \leq \varrho_0,\quad \|\varrho_0\|_{C^{2+\nu}(\mathbb T)}\leq \overline\varrho\qquad \mathbb P\mbox{-a.s.}
\end{equation*}
(see \eqref{ini.rho.smooth}).
In what follows, we use a notation $u(nh)$ for a piecewise constant function defined as $u(nh)(t) = u\left(\lfloor\frac tn\rfloor h\right)$ where $\lfloor\cdot\rfloor$ denotes the floor function. Although this notation might be confusing, its usage will be always clear.
Since
\begin{equation*}
\|[\vu(nh)]_{R}\|_{W^{2,\infty}(\mathbb T)}\leq c,
\end{equation*}
 we get, similarly to \eqref{eq:rho.reg}, that
\begin{equation}\label{eq:rhoh.bound}
{\rm esssup}_{t\in [0,T]}(\|\varrho_h(t)\|_{C^{2+\nu}(\mathbb T)} + \|\pat \varrho_h(t)\|_{C^\nu(\mathbb T)} + \|\varrho_h^{-1}(t)\|_{L^\infty(\mathbb T)})\leq c
\end{equation}
for a deterministic constant $c$ which is independent of $\varrho$, $\vu$ and $h$. By the same reasoning as above, 
\begin{equation*}
\varrho_h \in C^{\nu}([0,T],C^{2+\nu}(\mathbb T))
\end{equation*}
with norm independent of $h$.

To deduce the H\"older regularity of $\vu_h$ we take a test function $\varphi\in X_m$ in \eqref{eq:ap.prob3} and we integrate over $\mathbb T$ and a time interval $[\tau_1,\tau_2]\subset [0,T]$ to get
\begin{multline}\label{eq:uni.con}
\int_{\mathbb T} (\varrho_h \vu_h (\tau_2) - \varrho_h \vu_h(\tau_1))\varphi \ {\rm d}x = \int_{\tau_1}^{\tau_2} \int_{\mathbb T}(\varrho_h [\vu_h]_R\otimes \vu_h)\nabla \varphi \ {\rm d}x{\rm d}t\\
 - \int_{\tau_1}^{\tau_2}\int_{\mathbb T} \chi(\|\vu_h\|_{X_m}  - R) \nabla p_{\delta}(\varrho_h) \varphi\ {\rm d}x{\rm d}t
 + \int_{\tau_1}^{\tau_2}\int_{\mathbb T} \varepsilon \Delta(\varrho_h \vu_h)\varphi \ {\rm d}x{\rm d}t\\
 + \int_{\tau_1}^{\tau_2}\int_{\mathbb T}\diver S(D\vu_h) \varphi\ {\rm d}x{\rm d}t - \int_{\tau_1}^{\tau_2} \int_{\mathbb T} \varrho_h (\nabla K*\varrho_h)\varphi \ {\rm d}x{\rm d}t \\
+ \int_{\tau_1}^{\tau_2}\int_{\mathbb T} \left[\varrho_h(\psi*(\varrho_h \vu_h)) - \varrho_h \vu_h (\psi*\varrho_h)\right]\varphi\ {\rm d}x{\rm d}t
 + \int_{\tau_1}^{\tau_2} \int_{\mathbb T} \varrho_h \Pi_m (\mathbb F_\varepsilon (\varrho_h, \vu_h))\varphi \ {\rm d}x{\rm d}W.
\end{multline}
We take $\tau_1 = 0$ to get  
\begin{multline*}
\|\Pi_m(\varrho_h\vu_h)(\tau_2)\|_{X_m}^r\\
\leq c\left(\|\vu_h(0)\|_{X_m}^r + \int_{0}^{\tau_2} \sup_{0\leq s\leq t}\|\vu_h\|_{X_m}^r\ {\rm d}t + 1 + \left\|\int_{0}^{\tau_2} \Pi_m(\varrho_h \Pi_m(\mathbb F_\varepsilon(\varrho_h,\vu_h)))\ {\rm d}W\right\|^r_{X_m}\right)
\end{multline*}
where $c$ depends also on $R$. Since $\vu_h = \mathcal M^{-1}_{\varrho_h}(\Pi_m(\varrho_h \vu_h))$ and \eqref{eq:rhoh.bound} we deduce with the help of the Gronwall inequality that
\begin{equation}\label{eq:first.est}
\mathbb E\left[\sup_{\tau\in[0,T]} \|\Pi_m(\varrho_h \vu_h)(\tau)\|_{X_m}^r\right] + \mathbb E\left[\sup_{\tau\in [0,T]}\|\vu_h\|_{X_m}^r\right]
\leq c(1+\mathbb E\left[\|\vu_0\|^r_{X_m}\right])
\end{equation}
for every $r\geq 1$.

Further, \eqref{eq:uni.con} yields
\begin{multline*}
\|\Pi_m(\varrho_h \vu_h(\tau_2) - \varrho_h \vu_h(\tau_1))\|_{X_m}^r = \left(\sup_{\|\varphi\|_{X_m}\leq 1} \int_{\mathcal T}(\varrho_h \vu_h(\tau_2) - \varrho_h\vu_h(\tau_1))\varphi\ {\rm d}x\right)^r\\
\leq c\left(\int_{\tau_1}^{\tau_2} 1 \ {\rm d}t + \int_{\tau_1}^{\tau_2} \sup_{s\in[0,T]}\|\vu_h(s)\|_{X_m}\ {\rm d}t\right)^r + \left\|\int_{\tau_1}^{\tau_2} \Pi_m(\varrho_h \Pi_m(\mathbb F_\varepsilon(\varrho_h,\vu_h)))\ {\rm d}W\right\|^r_{X_m}.
\end{multline*}
The Burkholder-Davis-Gundy (for the separable Hilbert space variant see e.g. \cite{MR16}) inequality yields
\begin{multline}\label{Burkholder-davis-gundy}
\mathbb E\left[\left\|\int_{\tau_1}^{\tau_2} \Pi_m(\varrho_h \Pi_m(\mathbb F_\varepsilon(\varrho_h,\vu_h)))\ {\rm d}W\right\|^r_{X_m}\right]\\
\leq c\mathbb E \left[\left(\int_{\tau_1}^{\tau_2} \left\|\sum_{k=1}^{\infty}\Pi_m(\varrho_h \Pi_m(\mathbb F_{k,\varepsilon}(\varrho_h,\vu_h)))\right\|_{X_m}^2\ {\rm d}t\right)^{\frac r2}\right]\\
\leq c \mathbb E \left[\int_{\tau_1}^{\tau_2}\|\varrho_h (t)\|_{\Gamma}^2 \sum_{k=1}^\infty f_{k,\varepsilon}^2\ {\rm d}t\right]^{\frac r2}\leq c |\tau_2 - \tau_1|^{r/2}.
\end{multline}
due to the assumptions on $f_{k,\varepsilon}$. Thus, with help of \eqref{eq:first.est},
\begin{equation*}
\mathbb E\left[\|\Pi_m(\varrho_h \vu_h(\tau_2) - \varrho_h \vu_h(\tau_1))\|_{X_m}^r\right]\leq c|\tau_2 - \tau_1|^{r/2}\left((1+\mathbb E\left[\|\vu_0\|_{X_m}^r\right]\right)
\end{equation*}
for $r\geq 1$. Note that the constants $c$ in the aforementioned estimates are always independent of $h$ and $\varrho_h,\ \vu_h$. With help of the Kolmogorov continuity criterion (see e.g. \cite[Theorem 3.3]{DPZ92}) we deduce that there exists a modification of $\Pi_m(\varrho_h \vu_h)$ that has $\mathbb P$-a.s. $\beta$-H\"older continuous trajectories for all $\beta\in (0,(r-2)/2r)$. Since $\vu_h = \mathcal M^{-1}_{\varrho_h} (\Pi_m(\varrho_h \vu_h))$, we deduce that
\begin{equation*}
\mathbb E\left[\|\vu_h\|^r_{C^\beta((0,T),X_m)}\right]\leq c\left(1+ \mathbb E\left[\|\vu_0\|_{X_m}^r\right]\right)
\end{equation*}
with the constant $c$ independent of $h$ and the solution.

Now, we pass to a limit with $h$. 
Consider a path space for the basic state variables $[\varrho_h, \vu_h,W]$:
\begin{equation*}
\mathcal X=  \mathcal X_\varrho \times\mathcal X_\vu \times \mathcal X_W 
\end{equation*}
where
\begin{equation*}
\begin{split}
\mathcal X_\varrho & = C^\iota([0,T], C^{2+\iota})\\
\mathcal X_\vu & = C^\kappa([0,T],X_m)\\
\mathcal X_W & = C([0,T], \mathfrak{U}_0).
\end{split}
\end{equation*}
for some $0<\iota<\nu$ and $0<\kappa<\beta$.

The set of laws $\mathcal L(\varrho_h,\vu_h,W)$ is tight in $\mathcal X$. Indeed, Arzel\`a-Ascoli theorem and interpolation yield a compact embedding of $C^{2+\nu}(\mathbb T)$ to $C^{2+\iota}(\mathbb T)$ for $0<\iota<\nu$ and thus
$$
C^{\nu}([0,T], C^{2+\nu}(\mathbb T))\cap W^{1,\infty}([0,T], C^\nu(\mathbb T))\stackrel{c}{\hookrightarrow} C^\iota([0,T], C^{2+\iota}(\mathbb T)).
$$

Take 
\begin{multline*}B_c = \{\varrho  \in C^\nu([0,T],C^{2+\nu}(\mathbb T))\cap W^{1,\infty}([0,T],C^\nu(\mathbb T)),\\ \|\varrho\|_{C^\nu([0,T],C^{2+\nu}(\mathbb T))} + \|\varrho\|_{W^{1,\infty}([0,T],C^\nu(\mathbb T))} \leq c\}\end{multline*}
 where $c$ is the constant appearing in \eqref{eq:rhoh.bound}. Clearly, $\mathcal L(\varrho_h)(B_c) = 1$ and $B_c$ is compact in $\mathcal X_\varrho$ due to the aforementioned embedding. 

Further, $B_L = \{\vu\in C^\beta([0,T],X_m),\ \|\vu\|_{C^\beta([0,T],X_m)}\leq L\}$ is compact in $\mathcal X_\vu$ since $C^\beta([0,T],X_m)$ is compactly embedded into $C^\iota ([0,T],X_m)$ for $\iota<\beta$ due to the Arzel\`a-Ascoli theorem. By Chebychev's inequality,
\begin{multline}\label{eq:u.tight}
\mathcal L(\vu_h)(B_L) = 1-\mathcal L(\vu_h)(B_L^c) = 1-\mathbb P(\|\vu_h\|_{C^\beta([0,T],X_m)} >L)\\ \geq 1-\frac 1{L^r} \mathbb E(\|\vu_h\|_{C^\beta([0,T],X_m)}^r)
\geq 1-\frac c{L^r}
\end{multline}
and thus we get the claimed tightness.

Because $\mathcal{X}$ is not separable we cannot use the usual combination of Prokhorov and Skorokhod's theorems. Fortunately, there exists a sequence of continuous functions $g_n:\mathcal{X} \to(-1,1)$, $n\in\mathbb{N}$, that separates the points of $\mathcal{X}$.
Spaces enjoying this property are called sub-Polish (see \cite[Definition 2.1.3]{BrFeHo}), or quasi-Polish (see \cite[p. 4162]{BrOnSe16}).

The Jakubowski-Skorokhod theorem (see \cite[Theorem 2]{JAK97}) provides the similar outcome as the standard approach: There exist a complete probability space $(\tilde\Omega,\tilde \FF, \tilde{\mathbb P})$ and $\mathcal{X}$-valued Borel measurable random variables $(\tilde\varrho_h, \tilde \vu_h, \tilde W_h)$, $h\in (0,1)$ and $(\tilde \varrho, \tilde \vu, \tilde W)$ such that
\begin{itemize}
\item the law of $(\tilde \varrho_h, \tilde \vu_h, \tilde W_h)$ is given by $\mathcal L(\varrho_h, \vu_h, W)$,
\item the law of $(\tilde \varrho, \tilde \vu,  \tilde W)$ is a Radon measure on $\mathcal X$,
\item $(\tilde \varrho_h, \tilde \vu_h, \tilde W_h)$ converges $\tilde{\mathbb P}$-almost surely to $(\tilde \varrho, \tilde \vu, \tilde W)$ in the topology of $\mathcal X$. 
\end{itemize}

It is important to realize that $\tilde\varrho,\ \tilde\vu$ are classical stochastic processes unlike the ensuing parts of the construction where we will be forced to work with more general random distributions.

Lets denote the canonical filtration (for a definition see e.g. \cite[Remark 2.1.15]{BrFeHo}) of $(\tilde\varrho, \tilde\vu,\tilde W)$, where $\tilde W = (\tilde{W}_k)$, as
\[
\tilde\FF_t:= \sigma\left( \sigma_t[\tilde\varrho]\cup\sigma_t[\tilde\vu]\cup\bigcup_{k=1}^{\infty} \sigma_t[\tilde W_k]\right), \quad t\in[0,T].
\]

The trajectories of the processes $\tilde\varrho,\ \tilde\vu$ are $\tilde{\mathbb P}$-a.s. continuous and hence they are progressively measurable with respect to their canonical filtrations by \cite[Proposition 1.13]{KS91}) and thus also with respect to $(\tilde\FF_t)_{t\geq0}$.

The stochastic process $\tilde W$ is a cylindrical Wiener process with respect to its canonical filtration because its law is given by $\mathcal{L}(W)$ and $W$ is a cylindrical Wiener process. It can be shown (see \cite[p. 115]{BrFeHo}) that it is also a $(\tilde{\FF}_t)$-cylindrical process.

\begin{Lemma} The above constructed stochastic processes $\tilde\varrho,\ \tilde\vu$ together with a cylindrical Wiener process $\tilde W$ satisfy \eqref{eq:ap.prob} in the sense of distributions.
\end{Lemma}

\begin{proof}
First concern \eqref{eq:ap.prob}$_1$. By \cite[Theorem 2.9.1]{BrFeHo} the equation \eqref{eq:ap.prob2}$_1$ is satisfied also for $\tilde \varrho_h$ and $\tilde \vu_h$. The convergence in the topology of $\mathcal X$ is sufficient to deduce that $\tilde\varrho_1$ and $\tilde \vu_1$ is a weak solution to \eqref{eq:ap.prob}. In particular, $\tilde\varrho_h, \tilde \vu_h$ satisfy
\begin{equation*}
\int_0^T\int_{\mathbb T} \tilde \varrho_h \partial_t \varphi + \tilde \varrho_h [\tilde \vu_h]_R \nabla \varphi + \varepsilon \nabla \tilde \varrho_h \nabla \varphi \ {\rm d}x{\rm d}t = \int_{\mathbb T} \tilde \varrho_h\varphi (T)\ {\rm d}x - \int_{\mathbb T} \tilde \varrho_h \varphi(0)\ {\rm d}x
\end{equation*}
for all $\varphi \in C^\infty([0,T]\times \mathbb T)$. The convergence in $\mathcal X$ is sufficient to claim 
\begin{equation*}
\int_0^T\int_{\mathbb T} \tilde \varrho \partial_t \varphi + \tilde \varrho [\tilde \vu]_R \nabla \varphi  + \varepsilon \nabla \tilde \varrho \nabla \varphi \ {\rm d}x{\rm d}t = \int_{\mathbb T} \tilde \varrho \varphi(T)\ {\rm d}x - \int_{\mathbb T}\tilde \varrho \varphi(0)\ {\rm d}x.
\end{equation*}
It is worthwhile to mention that \eqref{eq:mass.cons} remains valis also for $\tilde\varrho$. 

It remains to handle \eqref{eq:ap.prob}$_2$. The convergences  of $\tilde\varrho_h$ and $\tilde\vu_h$ yield 
\begin{equation*}
\begin{split}
\nabla K*\tilde \varrho_h &\rightrightarrows \nabla K*\tilde \varrho\\
\psi*(\tilde\varrho_h \tilde \vu(nh))&\rightrightarrows \psi*(\tilde \varrho \tilde \vu)\\
\psi*\tilde \varrho_h & \rightrightarrows \psi*\tilde \varrho
\end{split}
\end{equation*}
$\mathbb P$-almost surely. 

The convergences of $\varrho_h$ and $\vu_h$ together with the continuity of coefficients of $\mathbb F_\varepsilon$ yield
\begin{equation*}
\Pi_m(\tilde\varrho_h \Pi_m(F_{k,\varepsilon}(\tilde\varrho_h(nh),\vu_h(nh))))\rightarrow \Pi_m (\tilde\varrho(\Pi_m(F_{k,\varepsilon}(\tilde\varrho,\tilde \vu)))\ \mbox{in }L^q([0,T]\times \mathbb T)
\end{equation*}
and
\begin{multline*}
\mathbb E \left(\int_0^T\|\Pi_m[\tilde\varrho_h \Pi_m[\mathbb F_\varepsilon(\tilde\varrho_h(nh),\tilde\vu_h(nh))]]\|_{L_2(\mathfrak U, L^2(\TT))}^2\ {\rm d}t\right) \\
\leq \sum_{k=1}^\infty\mathbb E \left(\int_0^T\|\Pi_m[\tilde\varrho_h \Pi_m[\mathbb F_{k,\varepsilon}(\tilde\varrho_h(nh),\tilde\vu_h(nh))]]\|_{L_2(\mathfrak U, L^2(\TT))}^2\ {\rm d}t\right) \\
\leq c \|\tilde \varrho_h \|_{L^\infty([0,T]\times \mathbb T)}^2 \sum_{k=1}^\infty f_{k,\varepsilon}^2 \leq c
\end{multline*}
and thus
\begin{equation}\label{eq:h:st}
\Pi_m(\tilde\varrho_h \Pi_m(\mathbb F_\varepsilon(\tilde\varrho_h(nh),\vu_h(nh))))\rightarrow \Pi_m (\tilde\varrho(\Pi_m(\mathbb F_\varepsilon(\tilde\varrho,\tilde \vu)))\ \mbox{in }L^2([0,T],L_2(\mathfrak U,L^2(\mathbb T)))
\end{equation}
$\mathbb P$-almost surely. Consequently, $\tilde \varrho_h, \tilde \vu_h$ satisfy \eqref{eq:ap.prob3} and we apply $\mathcal M^{-1}_{\tilde \varrho_h}$ to get
\begin{multline*}
\Pi_m(\tilde\varrho_h \tilde\vu_h)(t) = \Pi_m (\tilde \varrho_h\tilde \vu_h)(0) \\
+\int_0^t \Pi_m (-\diver(\tilde \varrho_h[\tilde \vu_h(nh)]_R\otimes \vu(nh)) - (\chi(\|\vu(nh)\|_{X_m} - R) \nabla p_\delta (\varrho))     )(s)\ {\rm d}s\\
+\int_0^t \Pi_m (\varepsilon \Delta (\tilde \varrho_h \vu(nh)) + \diver S(D\vu(nh)))(s)\ {\rm d}s\\
+\int_0^t \Pi_m (\tilde \varrho_h (\psi*(\tilde \varrho_h \tilde \vu_h(nh))) - \tilde \varrho_h \tilde \vu_h(nh)(\psi*\tilde \varrho_h) - \tilde \varrho_h (\nabla K*\tilde \varrho_h)(s)\ {\rm d}s\\
 + \int_0^t \Pi_m (\tilde \varrho_h \Pi_m (\mathbb F_\varepsilon(\tilde\varrho_h(nh),\tilde\vu_h(nh)))) \ {\rm d}\tilde W_h
\end{multline*}
The H\"older continuity of $\tilde \vu_h$ yields
\begin{equation*}
\|\tilde \vu_h(nh)(t) - \tilde \vu_h(t)\|_{X_m} \leq ch^\beta \| \tilde \vu_h\|_{C((0,T),X_m)}
\end{equation*}
and
\begin{equation*}
\|\tilde \varrho_h(nh)(t)  - \tilde \varrho_h(t)\|_{C^{2+\nu}(\TT)} \leq ch^\nu \|\tilde \varrho_h\|_{C((0,T),C^{2+\nu}(\TT))}
\end{equation*}
and we deduce that
\begin{multline*}
\Pi_m(\tilde\varrho_h \tilde\vu_h)(t) = \Pi_m (\tilde \varrho_h\tilde \vu_h)(0) \\
+\int_0^t \Pi_m (-\diver(\tilde \varrho_h[\tilde \vu_h]_R\otimes \vu) - (\chi(\|\vu\|_{X_m} - R) \nabla p_\delta (\varrho))     )(s)\ {\rm d}s\\
+\int_0^t \Pi_m (\varepsilon \Delta (\tilde \varrho_h \vu) + \diver S(D\vu))(s)\ {\rm d}s\\
+\int_0^t \Pi_m (\tilde \varrho_h (\psi*(\tilde \varrho_h \tilde \vu_h)) - \tilde \varrho_h \tilde \vu_h(\psi*\tilde \varrho_h) - \tilde \varrho_h (\nabla K*\tilde \varrho_h)(s)\ {\rm d}s\\
 + \int_0^t \Pi_m (\tilde \varrho_h \Pi_m (\mathbb F_\varepsilon(\tilde\varrho_h,\tilde\vu_h))) \ {\rm d}\tilde W_h + o(h).
\end{multline*}
Passing to a limit we obtain that $\tilde \varrho,\ \tilde \vu$ solves \eqref{eq:ap.prob}$_2$. The limit in the stochastic integral is justified by Lemma \ref{L266} because $\tilde W_h$ are cylindrical Wiener processes and $\tilde W_h \to \tilde W$ in $X_W = C([0,T], \mathfrak{U}_0)$ $\mathbb P$-a.s. and due to convergence in \eqref{eq:h:st}.

\end{proof}

\begin{Remark}
We say that $((\tilde \Omega, \tilde \FF, (\tilde{\FF})_{t\geq0},\tilde {\mathbb P}),\tilde \varrho, \tilde \vu, \tilde W)$ is a \emph{martingale solution} to \eqref{eq:ap.prob}.
\end{Remark}

Next we deduce the pathwise uniqueness result.
\begin{Lemma}\label{lem.unique}
Any two solutions to \eqref{eq:ap.prob} obtained by a previous method with the same initial cylindrical Wiener process are identical.
\end{Lemma}

\begin{proof}
Let $(\tilde\varrho_1,\tilde \vu_1, \tilde W)$ and $(\tilde\varrho_2,\tilde\vu_2,\tilde W)$ be two solutions to \eqref{eq:ap.prob}. We define a stopping time as follows
\begin{equation*}
\tau_M^i = \inf\{t\in[0,T],\ \|\tilde\varrho_i(t)\|_{C^{2+\nu}(\mathbb T)} + \|(\tilde\varrho_i)^{-1}(t)\|_{L^\infty(\mathbb T)} + \|\tilde\vu_i\|_{X_m}>M\|\}
\end{equation*}
where $M>0$. Further we set $\tau_M = \min\{\tau_M^1,\tau_M^2\}$. Chebyshev's inequality yields together with \eqref{eq:first.est}
\begin{equation*}
\mathbb P (\|\tilde\vu_i\|_{X_m} >M) \leq \frac 1M\mathbb E (\|\tilde\vu_i\|_{X_m})\leq \frac cM.
\end{equation*}
As a result, $\mathbb P(\sup_{M\in\mathbb N}\tau_M = T) = 1$.

We subtract equations for $\tilde \vu_1$ and $\tilde\vu_2$ and we multiply the difference by $\tilde\vu_1-\tilde\vu_2$ to obtain (with the help of Ito's product rule)
\begin{multline*}
\frac 12{\rm d} |\tilde\vu_1 - \tilde\vu_2| = (\tilde\vu_1 - \tilde\vu_2)(\mathcal M^{-1}_{\varrho_1} - \mathcal M^{-1}_{\varrho_2})[\Pi_m(-\diver (\tilde\varrho_1 [\tilde\vu_1]_R\otimes\tilde\vu_1)\\
 - \chi(\|\tilde\vu_1\|_{X_m} - R)\nabla p_\delta(\tilde\varrho_1) + \varepsilon\Delta(\tilde\varrho_1\tilde\vu_1) + \diver S(D\tilde\vu_1) + \tilde\varrho_1\psi*(\tilde\varrho_1\tilde\vu_1) - \tilde\varrho_1\tilde\vu_1(\psi*\tilde\varrho_1)\\
 - \tilde\varrho_1(\nabla K*\tilde\varrho_1) - \partial_t \tilde\varrho_1\tilde\vu_1)]{\rm d}t + (\tilde\vu_1 - \tilde\vu_2) \mathcal M^{-1}_{\varrho_2}[\Pi_m (\diver(\tilde\varrho_2[\tilde\vu_2]_R\otimes\tilde\vu_2) \\
- \diver(\tilde\varrho_1[\tilde\vu_1]_R\otimes\tilde\vu_1) + \chi(\|\tilde\vu_1\|_{X_m}-R)\nabla p_\delta(\tilde\varrho_1) - \chi(\|\tilde\vu_2\|_{X_m} - R)\nabla p_\delta(\tilde\varrho_2)\\
 + \varepsilon\Delta(\tilde\varrho_1\tilde\vu_1-\tilde\varrho_2\tilde\vu_2) + \diver S(D(\tilde\vu_1 - \tilde\vu_2)) + \tilde\varrho_1\psi*(\tilde\varrho_1\tilde\vu_1) - \tilde\varrho_2\psi*(\tilde\varrho_2\tilde\vu_2)\\
 + \tilde\varrho_2\tilde\vu_2 (\psi*\tilde\varrho_2) - \tilde\varrho_1\tilde\vu_1 (\psi*\tilde\varrho_1) + \tilde\varrho_2(\nabla K*\tilde\varrho_2) - \tilde\varrho_1(\nabla K*\tilde\varrho_1) + \partial_t \tilde\varrho_2\tilde\vu_2 - \partial_t\tilde\varrho_1\tilde\vu_1)]{\rm d}t\\
+ \Pi_m(\mathbb F(\tilde\varrho_1,\tilde\vu_1) - \mathbb F(\tilde \varrho_2,\tilde\vu_2)){\rm d}W + \frac 12 \sum_{k=1}^\infty |\Pi_m(F_{k,\varepsilon}(\tilde\varrho_1,\tilde\vu_1) - F_{k,\varepsilon}(\tilde\varrho_2,\tilde\vu_2))|{\rm d}t.
\end{multline*}
Recall that 
\begin{equation*}
\|\mathcal M^{-1}_{\tilde\varrho_1} - \mathcal M^{-1}_{\tilde\varrho_2}\|\leq c(\underline\varrho) \|\tilde\varrho_1 - \tilde\varrho_2\|_{L^1(\mathbb T)}
\end{equation*}
where $\underline \varrho$ is the common lower bound for $\tilde\varrho_1$ and $\tilde\varrho_2$ (see \eqref{eq:rhoh.bound}).
The definition of the stopping time yields
\begin{multline*}
\left|\int_{\mathbb T} (\tilde \vu_1 - \tilde \vu_2)(\mathcal M_{\tilde\varrho_1}^{-1} - \mathcal M_{\tilde\varrho_2}^{-1})(\Pi_m(\tilde\varrho_1\tilde\vu_1(\psi*\tilde\varrho_1)))\ {\rm d}x\right|\\
 + \left|\int_{\mathbb T}(\tilde \vu_1 - \tilde \vu_2)(\mathcal M_{\tilde\varrho_1}^{-1} - \mathcal M_{\tilde\varrho_2}^{-1}) (\Pi_m(\tilde\varrho_1 \psi*(\tilde\varrho_1\tilde\vu_1)))
) \ {\rm d}x\right|\\
 + \left|\int_{\mathbb T} (\tilde \vu_1 - \tilde \vu_2)(\mathcal M_{\tilde\varrho_1}^{-1} - \mathcal M_{\tilde\varrho_2}^{-1}) (\Pi_m (\tilde\varrho_1 (\nabla K*\tilde\varrho_1)
))\ {\rm d}x\right|\\
\leq c(\underline\varrho, M) \|\tilde\vu_1 - \tilde\vu_2\|_{X_m}\|\tilde\varrho_1 - \tilde\varrho_2\|_{L^\infty(\mathbb T)}
\end{multline*}
and
\begin{multline*}
\left|\int_{\mathbb T}(\tilde \vu_1 - \tilde\vu_2)\mathcal M_{\tilde\varrho_2}^{-1}(\Pi_m(\tilde\varrho_1\psi*(\tilde\varrho_1\tilde\vu_1) - \tilde\varrho_2\psi*(\tilde\varrho_2\tilde\vu_2)))\ {\rm d}x\right|\\
\left|\int_{\mathbb T}(\tilde \vu_1 - \tilde\vu_2)\mathcal M_{\tilde\varrho_2}^{-1}(\Pi_m(\tilde\varrho_1\tilde\vu_1(\psi*\tilde\varrho_1) - \tilde\varrho_2\tilde\vu_2(\psi*\tilde\varrho_2)))\ {\rm d}x\right|\\
\left|\int_{\mathbb T}(\tilde \vu_1 - \tilde\vu_2)\mathcal M_{\tilde\varrho_2}^{-1}(\tilde\varrho_1(\nabla K*\tilde\varrho_1) - \tilde\varrho_2(\nabla K*\tilde \varrho_2)
)))\ {\rm d}x\right|\\
\leq c(\underline\varrho,M) \left(\|\tilde\vu_1- \tilde\vu_2\|_{X_m}^2 + \|\tilde\vu_1 - \tilde\vu_2\|_{X_m}\|\tilde\varrho_1 - \tilde\varrho_2\|_{L^\infty(\mathbb T)}\right).
\end{multline*}
Most of other terms can be treated similarly. Due to known method which includes also the Gronwall inequality (see \cite[Corollary 4.10]{BrFeHo}) we infer that the estimate
\begin{multline*}
\mathbb E \left[\sup_{s[0,T]}(\|(\tilde\vu_1 - \tilde \vu_2)(s\wedge \tau_M)\|_{X_m}^2 + \|(\tilde\varrho_1 - \tilde\varrho_2)(s\wedge \tau_M)\|_{C^{2+\nu}(\mathbb T)}^2)\right]\\
\leq c \mathbb E\left[\|(\tilde\vu_1 - \tilde\vu_2)(0)\|_{X_m}^2 + \|(\tilde\varrho_1 - \tilde\varrho_2)(0)\|_{C^{2+\nu}(\mathbb T)}^2\right]
\end{multline*}
holds for some constant $c$ independent of $\tilde\vu$ and $\tilde\varrho$. This concludes the proof.
\end{proof}

%
%

\begin{Corollary}\label{cor.unique}
Let $\tilde \varrho$, $\tilde \vu$, $\tilde W$ are obtained as above.
Then we may assume that there is a triple $(\varrho,\vu,W)$ defined on
the original probability space $\Omega$ such that the laws of
$(\varrho,\vu)$ and $(\tilde \varrho,\tilde \vu)$ coincide, $W$ is the
given Wiener process, $(\varrho,\vu, W)$ is uniquely given and 
solves
\eqref{eq:ap.prob}.

\end{Corollary}

\begin{proof}
We consider a sequence of $5$-tuples
$$
\left\{(\varrho_{n_k},\vu_{n_k},\varrho_{m_k},\vu_{m_k},W)\right\}_{k=1}^\infty\subset
\mathcal X_\varrho \times \mathcal X_\vu\times \mathcal X_\varrho\times
\mathcal X_\vu \times \mathcal X_W
$$
where $\{n_k\}_{k=1}^\infty$ and $\{m_k\}_{k=1}^\infty$ are decreasing
sequences with zero limit. By similar arguments as above (with a key step being application of the Jakubowski-Skorokhod theorem on a sequence of joint laws generated by the above mentioned random variables) we deduce that there exist a complete probability space $(\tilde\Omega,\tilde\FF,\tilde{\mathbb{P}})$
and a sequence of Borel measurable random variables $(\tilde\varrho_{n_k},\tilde\vu_{n_k},\tilde\varrho_{m_k},\tilde\vu_{m_k},\tilde W)$ with the same law (passing to a subsequence as the case may be) satisfying
$$
(\tilde \varrho_{n_k},\tilde \vu_{n_k}, \tilde W_k)\to (\tilde
\varrho^1,\tilde \vu^1, \tilde W)
$$
and
$$
(\tilde \varrho_{m_k},\tilde \vu_{m_k}, \tilde W_k)\to (\tilde
\varrho^2,\tilde \vu^2, \tilde W)
$$
$\mathbb P$-almost surely in the corresponding spaces. As a consequence
of the previous lemma we obtain $(\tilde \varrho^1,\tilde \vu^1) =
(\tilde \varrho^2,\tilde \vu^2)$. Therefore we can apply a variant of the Gy\"ongy-Krylov theorem for sub-Polish spaces (see \cite[Theorem 2.10.3]{BrFeHo}) on the sequence $\{(\varrho_{n_k},\vu_{n_k})\}_{k=1}^\infty$ to deduce that it has a subsequence converging $\mathbb{P}$-a.s. to $(\varrho,\vu)$ defined on the original probability space $(\Omega,\FF,\mathbb{P})$. 
The equation \eqref{eq:ap.prob} is then satisfied by $(\varrho,\vu)$ with the initial Wiener process $W$.
\end{proof}

\subsection{Energy estimate}
Take $f(\varrho, \Pi_m(\varrho \vu))  = \frac 12 \int_{\mathbb T} \Pi_m(\varrho \vu) \mathcal M^{-1}_\varrho \Pi_m(\varrho \vu)\ {\rm d}x$. Since
\begin{equation*}
\begin{split}
\partial_{\Pi_m(\varrho\vu)} f(\varrho, \Pi_m(\varrho \vu))& = \mathcal M^{-1}_\varrho \Pi_m(\varrho\vu) = \vu\\
\partial^2_{\Pi_m(\varrho\vu)} f(\varrho, \Pi_m(\varrho \vu)) & = \mathcal M^{-1}_\varrho\\
\partial_\varrho f(\varrho, \Pi_m(\varrho \vu)) &= -\frac 12 \int_{\mathbb T} \Pi_m (\varrho \vu) \mathcal M^{-1}_\varrho \mathcal M_\cdot \mathcal M^{-1}_\varrho \Pi_m(\varrho\vu)\ {\rm d}x,
\end{split}
\end{equation*}
a direct application of Ito's formula and \eqref{eq:ap.prob} yield
\begin{multline*}
\frac 12 \int_{\mathbb T} \varrho \vu^2(T,\cdot) \Phi(T) \ {\rm d}x = \frac 12 \int_{\mathbb T} \varrho \vu^2(0,\cdot) \Phi(0)\ {\rm d}x - \frac 12 \int_0^T\int_{\mathbb T} (-\diver( \varrho [\vu]_R) + \varepsilon \Delta\varrho)\vu^2 \Phi \ {\rm d}x{\rm d}t\\
 + \int_0^T\int_{\mathbb T} \vu \left(  \varepsilon \Delta (\varrho \vu)  + \diver S(D \vu) - \varrho \nabla K*\varrho + \varrho( \psi*(\varrho \vu) - \vu(\psi*\varrho))\right)\Phi \ {\rm d}x{\rm d}t\\
 - \int_0^T\int_{\mathbb T}\vu\left(\diver(\varrho [\vu]_R\otimes \vu) + \chi(\|\vu\|_{X_m} - R)\nabla p_\delta (\varrho)\right)\Phi\ {\rm d}x{\rm d}t + \frac 12 \int_0^T \int_{\mathbb T} \varrho |\vu|^2 \partial_t \Phi\ {\rm d}x{\rm d}t\\
 + \int_0^T \int_{\mathbb T} \varrho \vu \Pi_m(\mathbb F_\varepsilon(\varrho,\vu)) \Phi\ {\rm d}x{\rm d}W + \frac 12 \int_0^T\int_{\mathbb T} \varrho (\Pi_m (\mathbb F_\varepsilon(\varrho,\vu)))^2\ {\rm d}x{\rm d}t
\end{multline*}
for all $\Phi \in C^\infty_c([0,T])$. An integrand $\int_{\mathbb T} \varrho \vu \Pi_m(\mathbb F_\varepsilon(\varrho,\vu))\ {\rm d}x$ of an stochastic integral is to be interpreted as an $L_2(\mathfrak{U},\mathbb{R})$-element, which is done by assigning:
\[
\sum_{k=1}^{\infty} \langle v, e_k\rangle e_k \in \mathfrak{U} \quad \mapsto \quad \sum_{k=1}^{\infty} \langle v, e_k\rangle  \int_{\mathbb T} \varrho \vu \Pi_m(\mathbf F_{k,\varepsilon}(\varrho,\vu))\ {\rm d}x \in \mathbb{R}. 
\]

After a suitable rearrangement, which includes also integration by parts and an application of the continuity equation, we arrive at
\begin{multline*}
\frac 12 \int_{\mathbb T} \varrho |\vu|^2(T,\cdot) \Phi(T)\ {\rm d}x  - \frac 12 \int_{\mathbb T}\varrho |\vu|^2(0,\cdot) \Phi(0)\ {\rm d}x + \int_0^T \int_{\mathbb T} S(D\vu):D\vu  \Phi\ {\rm d}x{\rm d}t\\
 + \varepsilon \int_0^T \int_{\mathbb T} \varrho |\nabla \vu|^2 \Phi \ {\rm d}x{\rm d}t + \frac 12 \int_{\mathbb T} \varrho K*\varrho(T,\cdot) \Phi(T)\ {\rm d}x - \frac 12 \int_{\mathbb T} \varrho_0 K*\varrho_0 \Phi(0)\ {\rm d}x \\ 
 + \frac12\int_0^T \int_{\mathbb T} \int_{\mathbb T} \varrho(t,x) \psi(x-y) \varrho(t,y) (\vu(t,y) - \vu(t,x))^2\Phi(t) \ {\rm d}x{\rm d}y {\rm d}t\\
 + \int_{\mathbb T} P_\delta (\varrho)(T,\cdot) \Phi(T) \ {\rm d} x - \int_{\mathbb T} P_{\delta}(\varrho_0) \Phi(0)\ {\rm d}x + \varepsilon \int_0^T\int_{\mathbb T} |\nabla \varrho|^2 P''_\delta (\varrho)\Phi\ {\rm d}x{\rm d}t\\
 = \frac 12 \int_0^T \int_{\mathbb T} \varrho K*\varrho \partial_t \Phi \ {\rm d}x {\rm d}t + \varepsilon \int_0^T \int_{\mathbb T} (\nabla K*\varrho )\nabla\varrho \Phi\ {\rm d}x{\rm d}t \\
+\int_0^T\int_{\mathbb T}(K*\varrho)\diver (\varrho (\vu - [\vu]_R))\Phi\ {\rm d}x{\rm d}t + \frac 12 \int_0^T\int_{\mathbb T} \varrho \vu^2 \partial_t \Phi \ {\rm d}x{\rm d}t + \int_0^T\int_{\mathbb T} P_\delta(\varrho) \partial_t\Phi\ {\rm d}x{\rm d}t\\
 + \int_0^T \int_{\mathbb T} \varrho \vu \Pi_m (\mathbb F_\varepsilon(\varrho,\vu)) \Phi \ {\rm d}x{\rm d}W + \frac 12 \int_0^T \int_{\mathbb T} \varrho (\Pi_m (\mathbb F_\varepsilon(\varrho,\vu)))^2\Phi \ {\rm d}x{\rm d}t.
\end{multline*}
We remind that $P_\delta$ is defined in \eqref{def.pdelta}.
Take a continuous function
\begin{equation*}
\Phi_{\tau,\kappa}(t) = \left\{
\begin{array}{l}
1,\ t\in [0,\tau-\kappa)\\
0,\ t\in (\tau,T]\\
\mbox{linear otherwise}
\end{array}
\right.
\end{equation*}
where $\tau\in (0,T]$ is arbitrary. The Lebesgue differentiation theorem allows to send $\kappa$ to $0$ in order to deduce

\begin{multline}\label{eq:ene.first}
\frac 12 \int_{\mathbb T} \varrho |\vu|^2(\tau,\cdot) \ {\rm d}x  - \frac 12 \int_{\mathbb T}\varrho |\vu|^2(0,\cdot)\ {\rm d}x + \int_0^\tau \int_{\mathbb T} S(D\vu):D\vu \ {\rm d}x{\rm d}t\\
 + \varepsilon \int_0^\tau \int_{\mathbb T} \varrho |\nabla \vu|^2  \ {\rm d}x{\rm d}t + \frac 12 \int_{\mathbb T} \varrho K*\varrho(\tau,\cdot) \ {\rm d}x - \frac 12 \int_{\mathbb T} \varrho K*\varrho (0,\cdot) \ {\rm d}x \\ 
 + \int_0^\tau \int_{\mathbb T} \int_{\mathbb T} \varrho(t,x) \psi(x-y) \varrho(t,y) (\vu(t,y) - \vu(t,x))^2\ {\rm d}x{\rm d}y {\rm d}t\\
 + \int_{\mathbb T} P_\delta (\varrho)(\tau,\cdot)  \ {\rm d} x - \int_{\mathbb T} P_{\delta}(\varrho)(0,\cdot)\ {\rm d}x + \varepsilon \int_0^\tau\int_{\mathbb T} |\nabla \varrho|^2 P''_\delta (\varrho)\ {\rm d}x{\rm d}t\\
 = \varepsilon \int_0^\tau \int_{\mathbb T} (\nabla K*\varrho )\nabla\varrho\ {\rm d}x{\rm d}t 
+\int_0^\tau\int_{\mathbb T}(K*\varrho)\diver (\varrho (\vu - [\vu]_R))\ {\rm d}x{\rm d}t  \\
+ \int_0^\tau \int_{\mathbb T} \varrho \vu \Pi_m (\mathbb F_\varepsilon(\varrho,\vu))\ {\rm d}x{\rm d}W +\frac 12 \int_0^\tau \int_{\mathbb T} \varrho (\Pi_m (\mathbb F_\varepsilon(\varrho,\vu)))^2 \  {\rm d}x{\rm d}t
\end{multline}
for almost all $\tau\in (0,T)$. Recall that $\|\varrho(\omega,t,\cdot)\|_{L^1(\TT)} = M$ due to \eqref{eq:mass.cons}. This together with $P''_\delta\geq c>0$ for all $\varrho\in [0,\infty)$ yield
\begin{equation*}
\varepsilon \int_{\mathbb T} \nabla K*\varrho \nabla \varrho \leq \frac 12 \varepsilon \int_\TT|\nabla \varrho |^2 P''_\delta (\varrho)\ {\rm d}x + c(K,M) \varepsilon.
\end{equation*}
Next, using the mass conservation law once again, we deduce 
\begin{multline*}
\left|\int_{\mathbb T} (K*\varrho) \diver(\varrho(\vu - [\vu]_R))\ {\rm d}x{\rm d}t\right| = \left|\int_{\mathbb T} (\nabla K*\varrho)(\varrho \vu - \varrho [\vu]_R)\ {\rm d}x\right|\\
\leq c(K,M)  \|\sqrt\varrho\|_{L^2(\TT)}\|\sqrt\varrho |\vu|\|_{L^2(\TT)}
\leq c(K,M) + \|\varrho|\vu|^2\|_{L^1(\TT)},
\end{multline*}
where we used the Young inequality in the very last step. 
Further, we have
\begin{equation*}
\int_{\mathbb T} \varrho (\Pi_m (F_\varepsilon(\varrho,  \vu)))^2\ {\rm d}x \leq \int_{\mathbb T}\varrho\sum_{k=1}^\infty f^2_{k,\varepsilon}\ {\rm d}x\leq c \|\varrho\|_{L^1(\TT)}
\end{equation*}
and
\begin{equation*}
\left| \int_{\mathbb T} \varrho \Pi_m (\mathbb F_{k,\varepsilon}(\varrho,\vu)) \vu \ {\rm d}x \right|^2 \leq f_{k,\varepsilon}^2 \|\varrho\|_{L^1(\TT)} \|\varrho\vu^2\|_{L^1(\TT)}.
\end{equation*}
Thus, it is possible to absorb some terms from the right hand side of \eqref{eq:ene.first} to the left hand side in order to deduce
\begin{multline*}
\frac 12 \int_{\mathbb T} \varrho |\vu|^2(\tau,\cdot) \ {\rm d}x  - \frac 12 \int_{\mathbb T}\varrho |\vu|^2(0,\cdot)\ {\rm d}x + \int_0^\tau \int_{\mathbb T} S(D\vu):D\vu \ {\rm d}x{\rm d}t\\
 + \varepsilon \int_0^\tau \int_{\mathbb T} \varrho |\nabla \vu|^2  \ {\rm d}x{\rm d}t + \frac 12 \int_{\mathbb T} \varrho K*\varrho(\tau,\cdot) \ {\rm d}x - \frac 12 \int_{\mathbb T} \varrho K*\varrho (0,\cdot) \ {\rm d}x \\ 
 + \int_0^\tau \int_{\mathbb T} \int_{\mathbb T} \varrho(t,x) \psi(x-y) \varrho(t,y) (\vu(t,y) - \vu(t,x))^2\ {\rm d}x{\rm d}y {\rm d}t\\
 + \int_{\mathbb T} P_\delta (\varrho)(\tau,\cdot)  \ {\rm d} x - \int_{\mathbb T} P_{\delta}(\varrho)(0,\cdot)\ {\rm d}x + \varepsilon \int_0^\tau\int_{\mathbb T} |\nabla \varrho|^2 P''_\delta (\varrho)\ {\rm d}x{\rm d}t\\
 \leq c(K,M) + \varepsilon c(K,M) +  c(K,M)\int_0^\tau \|\varrho |\vu|^2\|_{L^1(\TT)}\ {\rm d}t
 + \int_0^\tau \int_{\mathbb T} \varrho \vu \Pi_m (\mathbb F_\varepsilon(\varrho,\vu))\ {\rm d}x{\rm d}W.
\end{multline*}
The Burkholder-Davis-Gundy inequality yields 
\begin{equation*}
\mathbb E \left[\sup_{0\leq t\leq \tau}\left|\int_0^t \int_{\mathbb T} \varrho \Pi_m(\mathbb F_\varepsilon(\varrho, \vu)) \vu \ {\rm d}x{\rm d}W\right|^r\right] \leq \mathbb E \left[ \int_0^\tau \sum_{k=1}^\infty \left| \int_{\mathbb T} \varrho \Pi_m (\mathbb F_{k,\varepsilon}(\varrho,\vu)) \vu \ {\rm d}x\right|^2 \ {\rm d}t\right]^{\frac r2}.
\end{equation*}
Recall also that $|P_\delta(\varrho)|\geq c\left(|\varrho|^2 + |\varrho|^6\right)$. We use the Gronwall inequality to infer
\begin{multline*}
\mathbb E \left[\left|\sup_{\tau \in [0,T]}\int_{\mathbb T} \left[\frac 12 \varrho |\vu|^2+\frac 12 \varrho K*\varrho + P_\delta(\varrho)\right](\tau)\ {\rm d}x \right|^r\right]\\
 + \mathbb E \left[\left|\int_0^\tau \int_{\mathbb T} \left[S(D\vu):D\vu + \varepsilon \varrho |\nabla \vu|^2 + \varepsilon |\nabla \varrho|^2 P''_\delta(\varrho)\right]\ {\rm d}x{\rm d}t\right|^r\right]\\
 + \mathbb E \left[\left|\frac12\int_0^\tau \int_{\mathbb T}\int_{\mathbb T} \varrho(t,x)\varrho(t,y) \psi(x-y)(\vu(t,x) - \vu(t,y))^2\ {\rm d}x{\rm d}y{\rm d}t\right|^r\right]\\
\leq c(K,M,init.value)
\end{multline*}

for almost all $\tau \in (0,T)$. 

\subsection{Limit in the Galerkin approximations}

Let $\varrho_R$ and $\vu_R$ be the solution constructed in the previous section corresponding to a certain value $R\in \mathbb R$.
We define a stopping time 
\begin{equation*}
\tau_R = \inf \{\tau \in [0,T],\ \|\vu_R(t)\|_{X_m} >R\}.
\end{equation*}
The solution is determined uniquely up to $\tau_R$ (c.f. Lemma \ref{lem.unique} and Corollary \ref{cor.unique}), i.e. let $R_1<R_2$ and let $\varrho_{R_1},\vu_{R_1}$ and $\varrho_{R_2}$ and $\vu_{R_2}$ be the corresponding solutions, then $(\varrho_{R_1},\vu_{R_1}) = (\varrho_{R_2},\vu_{R_2})$ for $t\in (0,\tau_{R_1})$. Therefore, we may skip index $R$ in our notation. To prove that $\mathbb P(\sup_{R\in \mathbb N} \tau_R = T) = 1$, we recall that 
\begin{equation*}
\underline \varrho e^{-\int_0^\tau\|\nabla \vu\|_\infty\ {\rm d}t} \leq \varrho (\tau,x).
\end{equation*}
Consequently, since $\|\sqrt\varrho \vu(t)\|_{L^2}\leq c$ for some $R$-independent constant $c$ we get
\begin{equation*}
\mathbb E \left[e^{-c\int_0^\tau \|\nabla \vu\|_{L^2(\mathbb T)}^2\ {\rm d}t} \sup_{0\leq t\leq \tau} \|\vu\|_{L^2(\mathbb T)}^2\right]\leq c.
\end{equation*}
Further, we define sequences $a_R$ and $b_R$ such that $a_R\to \infty$, $b_R\to \infty$ as $R\to \infty$ and, simultaneously, $a_Re^{b_R} = R$ for every $R\in \mathbb N$. We introduce
\begin{equation*}
\begin{split}
A & =  \left[e^{-c\int_0^\tau \|\nabla \vu\|_{L^2(\mathbb T)}^2\ {\rm d}t} \sup_{0\leq t\leq \tau} \|\vu\|_{L^2(\mathbb T)}^2\leq a_R\right]\\
B & =  \left[c\int_0^\tau \|\nabla \vu\|_{L^2(\mathbb T)}^2\ {\rm d}t\leq b_R\right]\\
C & =  \left[\sup_{0\leq t\leq \tau} \|\vu\|_{L^2(\mathbb T)}^2\leq a_R e^{b_R}\right].
\end{split}
\end{equation*}
Consequently, we have
$$
\mathbb P(A)\geq 1-\frac c{a_R},\qquad \mathbb P(B) \geq 1- \frac c{b_R}
$$
and since $A\cap B \subset C$
$$
\mathbb P(C) \geq \mathbb{P}(A) + \mathbb{P}(B) - 1 \geq 1-\frac c{a_R} - \frac c{b_R} \to 1\ \mbox{as }R\to \infty
$$
which yields the demanded claim.

As a result, we get $\vu\in C(0,T,X_m)$ and $\varrho \in C^{\nu}(0,T, C^{2+\nu}(\mathbb T))$ satisfying the equation
\begin{equation}
\begin{split}\label{eq:po.R}
{\rm d} \varrho  + \diver (\varrho \vu)\ {\rm d}t & = \varepsilon \Delta \varrho {\rm d}t\\
{\rm d} \Pi_m(\varrho \vu)  + \Pi_m(\diver (\varrho \vu\otimes\vu)){\rm d}t + \Pi_m(\nabla p_\delta(\varrho)){\rm d} t &=  \Pi_m(\varepsilon \Delta (\varrho \vu)){\rm d}t
 + \Pi_m(\diver S(D\vu)){\rm d}t\\ &  - \Pi_m (\varrho (\nabla K*\varrho)){\rm d}t\\
& + \Pi_m (\varrho (\psi*(\varrho \vu) - u(\psi*\varrho)))){\rm d}t\\
& + \Pi_m (\varrho \Pi_m (F_\varepsilon(\varrho,\vu))){\rm d}W
\end{split}
\end{equation}

Our next goal is to proceed with $m$ to infinity. Note that $\varrho$ and $\vu$ constructed by the previous step satisfy \eqref{eq:ene.first} 
 without the second term on the right hand side, i.e.,
\begin{multline}\label{eq:ene.thirtyfour}
\frac 12 \int_{\mathbb T} \varrho |\vu|^2(T,\cdot) \Phi(T)\ {\rm d}x  - \frac 12 \int_{\mathbb T}\varrho |\vu|^2(0,\cdot) \Phi(0)\ {\rm d}x + \int_0^T \int_{\mathbb T} S(D\vu):D\vu  \Phi\ {\rm d}x{\rm d}t\\
 + \varepsilon \int_0^T \int_{\mathbb T} \varrho |\nabla \vu|^2 \Phi \ {\rm d}x{\rm d}t + \frac 12 \int_{\mathbb T} \varrho K*\varrho(T,\cdot) \Phi(T)\ {\rm d}x - \frac 12 \int_{\mathbb T} \varrho_0 K*\varrho_0 \Phi(0)\ {\rm d}x \\ 
 + \frac12\int_0^T \int_{\mathbb T} \int_{\mathbb T} \varrho(t,x) \psi(x-y) \varrho(t,y) (\vu(t,y) - \vu(t,x))^2\Phi(t) \ {\rm d}x{\rm d}y {\rm d}t\\
 + \int_{\mathbb T} P_\delta (\varrho)(T,\cdot) \Phi(T) \ {\rm d} x - \int_{\mathbb T} P_{\delta}(\varrho_0) \Phi(0)\ {\rm d}x + \varepsilon \int_0^T\int_{\mathbb T} |\nabla \varrho|^2 P''_\delta (\varrho)\Phi\ {\rm d}x{\rm d}t\\
 \leq \frac 12 \int_0^T \int_{\mathbb T} \varrho K*\varrho \partial_t \Phi \ {\rm d}x {\rm d}t + \varepsilon \int_0^T \int_{\mathbb T} (\nabla K*\varrho )\nabla\varrho \Phi\ {\rm d}x{\rm d}t \\
+ \frac 12 \int_0^T\int_{\mathbb T} \varrho \vu^2 \partial_t \Phi \ {\rm d}x{\rm d}t + \int_0^T\int_{\mathbb T} P_\delta(\varrho) \partial_t\Phi\ {\rm d}x{\rm d}t\\
 + \int_0^T \int_{\mathbb T} \varrho \vu \Pi_m (\mathbb F_\varepsilon(\varrho,\vu)) \Phi \ {\rm d}x{\rm d}W + \frac 12 \int_0^T \int_{\mathbb T} \varrho (\Pi_m (\mathbb F_\varepsilon(\varrho,\vu)))^2\Phi \ {\rm d}x{\rm d}t
\end{multline}
is valid for all $\Phi\in C^\infty([0,T])$, $\Phi\geq 0$. 
 Let $(\varrho_m,\vu_m)$ be a solution to \eqref{eq:po.R} in a space $X_m,\ m\in \mathbb N$. Since $(\varrho_m,\vu_m)$ satisfies \eqref{eq:ene.thirtyfour}, we deduce the following bound independent of $m$ (recall \eqref{viscosity.ass})
\begin{multline}
\mathbb E \left[(\esssup_{t\in[0,T]} \|\varrho_m |\vu_m|^2(t,\cdot)\|_{L^1(\mathbb T)})^r\right] + \mathbb E\left[(\|\nabla \vu_m\|_{L^2((0,T),L^2(\mathbb T))})^{2r}\right]\\ + \mathbb E\left[(\esssup_{t\in[0,T]} \|\varrho_m\|^\Gamma_{L^\Gamma(\mathbb T)})^r\right]
\leq c\label{eq:supremum}
\end{multline}
where $r>2$ is chosen later. 
Note that the periodic boundary condition imply $\int_\TT \nabla \vu_m \ {\rm d}x = 0$ and thus
$$
\|\nabla \vu_m\|_{L^2(\TT)} \leq c \|D\vu_m \|_{L^2(\TT)}
$$
(see e.g. \cite[Theorem 5.17]{DiRuSc}). Next, we have due to the Poincar\'e inequality and the mass conservation law that\footnote{For a function $f:\TT\mapsto \mathbb R$ we define its average as $(f)_\TT := \int_\TT f \ {\rm d}x$.}
\begin{multline*}
\|\vu_m\|_{L^2(\TT)} \leq \|\vu_m - (\vu_m)_\TT\|_{L^2(\TT)} + |(\vu_m)_\TT|\leq \|D\vu_m\|_{L^2(\TT)} + \frac 1M \int_\TT \varrho_m |(\vu_m)_\TT|\ {\rm d}x\\
 \leq \|D\vu_m \|_{L^2(\TT)}  + \frac 1M\int_\TT \varrho_m |\vu_m - (\vu_m)_\TT|\ {\rm d}x +\frac 1M \int_\TT \varrho_m |\vu_m|\ {\rm d}x\\
\leq c \left(\|D\vu_m\|_{L^2(\TT)} + 
\|\varrho_m\|_{L^\Gamma(\TT)} \|D\vu_m\|_{L^2(\TT)} + \|\varrho\|_{L^1(\TT)}^{\frac 12} \|\varrho_m|\vu_m|^2\|_{L^1(\TT)}^{\frac 12}\right)
\end{multline*}
where $M = (\varrho_m)_\TT$ is constant with respect of $t$ and $m$. This yields that
\begin{equation}\label{eq:odhad.rychlosti}
\mathbb E \left[\|\vu_m\|_{L^2((0,T),W^{1,2}(\TT))}^{ 2 r}\right]\leq c.
\end{equation}
By interpolation we deduce
\begin{equation*}
\|\varrho \vu \|_{L^{\frac{2\Gamma}{\Gamma+1}}(\mathbb T)} \leq c \|\sqrt \varrho\|_{L^{2\Gamma}(\mathbb T)}\|\sqrt \varrho \vu\|_{L^2(\mathbb T)}
\end{equation*}
and thus
\begin{equation*}
\mathbb E \left[(\esssup_{t\in[0,T]} \|\varrho \vu(t,\cdot)\|_{L^{\frac{2\Gamma}{\Gamma+1}}(\mathbb T)})^{\frac{2\Gamma}{\Gamma+1}r}\right]\leq c.
\end{equation*}
We have
\begin{equation*}
\|\varrho \vu\|_{L^1((0,T),L^3(\mathbb T))} \leq \|\varrho \|_{L^\infty((0,T), L^6(\mathbb T))} \|\vu \|_{L^1((0,T),L^6(\mathbb T))}.
\end{equation*}
By interpolation
\begin{equation*}
\|\varrho \vu\|_{L^{\frac{7\Gamma + 3}{3\Gamma + 3}}((0,T)\times \mathbb T)} \leq \|\varrho\vu\|_{L^1((0,T),L^3(\mathbb T))}^{\frac{3\Gamma + 3}{7\Gamma +3}}\|\varrho \vu\|_{L^\infty((0,T), L^{\frac{2\Gamma}{\Gamma+1}}(\mathbb T))}^{\frac{4\Gamma}{7\Gamma + 3}}
\end{equation*}
By known results on the regularity to parabolic equations (see \cite{DeHiPr}) we deduce from the continuity equation that
\begin{equation}\label{eq:derivace.rho}
\mathbb E\left[\|\varrho\|_{L^{\frac{7\Gamma + 3}{3\Gamma + 3}}((0,T),W^{1,\frac{7\Gamma + 3}{3\Gamma + 3}}(\mathbb T))}\right] \leq c
\end{equation}
and, consequently, also
\begin{equation}\label{eq:rho2}
\mathbb E\left[\|\pat \varrho\|_{L^{q}((0,T)\times \mathbb T)} + \|\varrho\|_{L^{q}((0,T),W^{2,q}(\mathbb T))}\right]\leq c.
\end{equation}
for certain $q>1$. 

It is also worth mentioning that \eqref{eq:derivace.rho}, \eqref{eq:odhad.rychlosti} and the regularity theory for the parabolic equations also yield
\begin{equation*}
\mathbb E(\|\varrho\|_{L^1((0,T),L^\infty(\mathbb T))})\leq c.
\end{equation*}

Next we consider the path space
\begin{equation*}
\mathcal X = \mathcal X_{\varrho_0}\times \mathcal X_{\vu_0}\times\mathcal X_{\varrho}\times\mathcal X_{\varrho\vu}\times\mathcal X_{\vu}\times\mathcal X_W,
\end{equation*}
where
\begin{equation*}
\begin{split}
&\mathcal X_{\varrho_0} = C(\mathbb T),\ \mathcal X_{\vu_0} = L^2(\mathbb T),\ \mathcal X_\vu = (L^2((0,T), W^{1,2}(\mathbb T)),w),\ \mathcal X_W = C([0,T], \mathfrak U_0),\\
&\mathcal X_\varrho = (L^{\frac{7\Gamma + 3}{3\Gamma + 3}}((0,T), W^{1,\frac{7\Gamma + 3}{3\Gamma + 3}}(\mathbb T))\cap \\
&\quad W^{1,q}((0,T),L^{q}(\mathbb T))\cap L^{q}((0,T),W^{2,q}(\mathbb T)),w) \cap C_w([0,T],L^\Gamma(\mathbb T)),\\
&\mathcal X_{\varrho \vu} = C([0,T], W^{-k,2}(\mathbb T))\cap C_w([0,T], L^{\frac{2\Gamma}{\Gamma+1}}(\mathbb T)),
\end{split}
\end{equation*}
for sufficiently high integer $k$. Here $C_w(X,Y)$ denotes the space of functions from $X$ to $Y$ continuous with respect to the weak topology of $Y$ and $(X,w)$ denotes space $X$ considered with the weak topology. 

We claim that the set of joint laws
\begin{equation*}
\{\mathcal L (\varrho_0,\vu_{0,m}, \varrho_m, \Pi_m(\varrho_m\vu_m), \vu_m ,W),\ m\in \mathbb N\}
\end{equation*}
is tight on $\mathcal X$. 
Indeed, $\mathcal L[\vu_m]$ is tight on $\mathcal X_\vu$ since the set 
$$
B_L = \{\vu \in L^2((0,T),W^{1,2}(\mathbb T)),\ \|\vu\|_{L^2((0,T),W^{1,2}(\mathbb T))}\leq L\}
$$
is compact in $\mathcal X_\vu$ and 
$$
\mathcal L[\vu_m] (B_L^c) = \mathbb P (\|\vu_m\|_{L^2((0,T),W^{1,2}(\mathbb T))}\geq L)\leq \frac 12 \mathbb E(\|\vu\|_{L^2((0,T),W^{1,2}(\mathbb T))})\leq \frac CL
$$
according to Chebyshev's inequality.

Further, the law $\mathcal L[\varrho_m]$ is tight on $\mathcal X_\varrho$. Indeed, \eqref{eq:supremum} and the continuity equation yields $\mathbb E(\|\varrho_m\|_{C^{0,1}([0,T],W^{-2,12/7}(\mathbb T))})\leq C$. Further, \cite[Lemma 6.2]{NoSt} yields
$$
L^\infty((0,T),L^6(\mathbb T))\cap C^{0,1}([0,T],W^{-2,\frac{2\Gamma}{\Gamma+1}}(\mathbb T))\subset\subset C_w([0,T],L^\Gamma(\mathbb T)).
$$
Thus, $B_L = \{\varrho\in C(L^\Gamma),\ \|\varrho\|_{C^{0,1}([0,T],W^{-2,\frac{2\Gamma}{\Gamma+1}}(\mathbb T))} + \|\varrho\|_{L^\infty((0,T), L^\Gamma(\mathbb T))}\leq L\}$ is compact in $C_w([0,T],L^\Gamma(\mathbb T))$ and $\mathcal L_\varrho (B_L)\geq 1-\frac cL$ due to Chebyshev's inequality.

The tightness of the law $\mathcal L(\varrho_m)$ in $(L^{\frac{7\Gamma + 3}{3\Gamma + 3}}((0,T), W^{1,\frac{7\Gamma + 3}{3\Gamma + 3}}(\mathbb T))\cap W^{1,q}((0,T),L^{q}(\mathbb T))\cap L^{q}((0,T),W^{2,q}(\mathbb T)),w) $ follows immediately from \eqref{eq:rho2}. Next, it is known that $L^\infty((0,T),L^\Gamma(\TT))\cap W^{1,q}((0,T),L^q(\TT))$ is compactly embedded into $C_w([0,T],L^\Gamma)$. This is a consequence of the Arzel\`a-Ascoli theorem (c.f. \cite[Lemma 6.2]{NoSt})\smallskip

Further, $\mathcal L[\Pi_m(\varrho_m\vu_m)]$ is tight in $\mathcal X_{\varrho \vu}$. Indeed, we have
\begin{equation*}
\Pi_m(\varrho_m \vu_m)(\tau) = Y_m(\tau) + \int_0^\tau \Pi_m(\varrho_m \Pi_m(\mathbb F_\varepsilon(\varrho_m, \vu_m)))\ {\rm d}W
\end{equation*}
where 
\begin{multline*}
Y_m(\tau) = \Pi_m(\varrho_m\vu_m)(0) + \int_0^\tau\Pi_m( \varepsilon \Delta(\varrho_m\vu_m))\ {\rm d}t + \int_0^\tau\Pi_m(\diver S(D\vu_m))\ {\rm d}t\\
 - \int_0^\tau\Pi_m(\varrho_m(\nabla K*\varrho_m))\ {\rm d}t + \int_0^\tau\Pi_m(\varrho_m(\psi*(\varrho_m\vu_m) - \vu_m (\psi*\varrho_m)))\ {\rm d}t\\
 - \int_0^\tau\Pi_m(\diver(\varrho_m\vu_m\otimes\vu_m))\ {\rm d}t - \int_0^\tau \Pi_m(\nabla p_\delta(\varrho_m))\ {\rm d}t.
\end{multline*}
We can write $Y_m(\tau) = \Pi_m(\varrho_m\vu_m)(0) + \int_0^\tau Z_m\ {\rm d}t$ and the already deduced bounds \eqref{eq:supremum} and \eqref{eq:odhad.rychlosti} yield $\mathbb E(\|Z_m\|^{r/2}_{L^2((0,T),W^{-3,2}(\TT))})\leq c$. This is enough to obtain 
$$\mathbb E \left[\|Y_m\|_{C^{1/2}((0,T),W^{-3,2}(\TT))}\right]\leq c$$
uniformly in $m$ assuming the initial data are chosen appropriately.\\
Further, we have by the Burkholder-Davis-Gundy inequality (compare with \eqref{Burkholder-davis-gundy})
\begin{equation*}
\mathbb E \left[\left\|\int_0^\cdot \varrho_m \Pi_m (\mathbb F_\varepsilon(\varrho_m,\vu_m))\ {\rm d}W\right\|_{C^{\kappa}([0,T],L^2(\TT))}^r\right]\leq c.
\end{equation*}
The compact embedding
\begin{equation*}
C^\kappa([0,T],W^{-l,2}(\TT))\stackrel{c}{\hookrightarrow} C ([0,T],W^{-k,2}(\TT))
\end{equation*}
for $k>l$ then yields tightness in $C([0,T],W^{-k,2}(\TT))$. The tightness in $C_w([0,T],L^{2\Gamma/(\Gamma + 1)}(\TT))$ then follows from the compact embedding
$$
L^\infty((0,T),L^{\frac{2\Gamma}{(\Gamma +  1)}}(\TT))\cap C^\kappa([0,T],W^{-l,2}(\TT)) \stackrel{c}{\hookrightarrow} C_w([0,T],L^{\frac{2\Gamma}{(\Gamma + 1)}}(\TT))
$$
which can be found in \cite{NoSt}.

Tightnesses of $\mathcal L(\varrho_0)$, $\mathcal L (\vu_{0,m})$ and $\mathcal L(W)$ follow by simple standard arguments.

As a matter of fact, the Jakubowski-Skorokhod theorem yields the existence of a complete probability space $(\tilde \Omega, \tilde \FF, \tilde{\mathbb{P}})$ with $\mathcal{X}$-valued Borel variables
$$
(\tilde \varrho_{0,m}, \tilde \vu_{0,m}, \tilde \varrho_m, \widetilde{\varrho_m\vu_m}, \tilde \vu_m,\tilde W_m)
$$ 
and
$$
(\tilde \varrho_{0}, \tilde \vu_{0}, \tilde \varrho, \widetilde{\varrho\vu}, \tilde \vu,\tilde W)
$$
such that
\begin{enumerate}
\item The laws of $
(\tilde \varrho_{0,m}, \tilde \vu_{0,m}, \tilde \varrho_m, \widetilde{\varrho_m\vu_m}, \tilde \vu_m,\tilde W_m)
$
and
$
( \varrho_{0,m},  \vu_{0,m},  \varrho_m, {\varrho_m\vu_m},  \vu_m, W_m)
$
coincide on $\mathcal X$. 
\item The law of $
(\tilde \varrho_{0}, \tilde \vu_{0}, \tilde \varrho, \widetilde{\varrho\vu}, \tilde \vu,\tilde W)
$ on $\mathcal X$ is a Radon measure.
\item The following convergences hold true $\tilde{\mathbb P}$-almost surely.
\begin{equation}\label{eq.nonmom.conv}
\begin{split}
\tilde \varrho_{0,m}&\to \tilde\varrho_0\ \mbox{in }C(\mathbb T)\\
\tilde \vu_{0,m}&\to \tilde \vu_0 \ \mbox{in }L^2(\mathbb T)\\
\tilde \varrho_m&\to \tilde \varrho \ \mbox{in }L^{\frac{7\Gamma + 3}{3\Gamma + 3}}((0,T),W^{1,\frac{7\Gamma + 3}{3\Gamma + 3}}(\TT))\\
\tilde \varrho_m&\to \tilde \varrho \ \mbox{weakly in } W^{1,q}((0,T),L^{q}(\TT))\cap L^{q}((0,T),W^{2,q}(\TT))\\
\tilde \varrho_m&\to \tilde \varrho \ \mbox{in }C_w((0,T),L^\Gamma(\TT))\\
\widetilde{\varrho_m\vu_m}&\to \widetilde{\varrho \vu}\ {\mbox{in }}C((0,T),W^{-k,2}(\TT))\ \mbox{for some }k\in \mathbb N\\
\widetilde{\varrho_m\vu_m}&\to \widetilde{\varrho \vu}\ \mbox{in }C_w((0,T),L^{\frac{7\Gamma + 3}{3\Gamma + 3}}(\TT))\\
\tilde \vu_m&\to \tilde \vu \ \mbox{weakly in }L^2((0,T),W^{1,2}(\TT))\\
\tilde W_m &\to \tilde W\ \mbox{in }C([0,T],\mathfrak U_0). 
\end{split}
\end{equation}
\end{enumerate}

It should be note that an element $\tilde\vu$ obtained as a limit above is to be interpreted as a random variable with the values in $L^2((0,T),W^{1,2}(\TT))$. It is no longer a classical stochastic process since since the values $\tilde\vu(t,\cdot)$ are not always defined. It becomes useful to introduce a concept of \emph{random variables} -- a generalization of the stochastic processes suitable for this context. A formal definition and properties can be found in \cite[Section 2.2]{BFH17}. For us, it will be important that under certain hypothesis, for a random distribution there exists a progressively measurable stochastic process in the same equivalence class.

By the same arguments as in Section~\ref{S:2.2} above it can be shown that $\tilde W$ is a cylindrical Wiener with respect to
\[
\tilde\FF_t:= \sigma\left( \sigma_t[\tilde\varrho]\cup\sigma_t[\tilde\vu]\cup\bigcup_{k=1}^{\infty} \sigma_t[\tilde W_k]\right), \quad t\in[0,T]
\]
and the corresponding stochastic integral can be well-defined.

As $\tilde \varrho_m \tilde \vu_m \to \tilde \varrho \tilde \vu$ weakly in $L^1((0,T),L^1(\mathbb T))$ $\tilde{\mathbb P}$-a.s., the equality of joint laws
\begin{equation*}
(\varrho_m, \vu_m, \Pi_m(\varrho_m\vu_m))\ \mbox{and } (\tilde \varrho_m, \tilde \vu_m, \widetilde{\varrho_m\vu_m})
\end{equation*}
allows to deduce
\begin{equation}\label{eq.mom.conv}
\widetilde{\varrho_m\vu_m} = \Pi_m (\tilde \varrho_m\tilde\vu_m)\ \mbox{and } \widetilde{\varrho \vu} = \tilde \varrho\tilde \vu.
\end{equation}

It remains to perform limits in the corresponding equations (i.e., \eqref{eq:po.R}). First, let us mention that all the 'convolution' terms converge to its counterparts uniformly, i.e., 
\begin{equation*}
\begin{split}
\nabla K*\tilde\varrho_m &\rightrightarrows \nabla K*\tilde\varrho\\
\psi*\tilde\varrho_m &\rightrightarrows \psi*\tilde\varrho.
\end{split}
\end{equation*}
$\tilde{\mathbb P}$-almost surely. Indeed, let prove (for instance) the second convergence. First, since $\tilde\varrho_m\to\tilde\varrho$ strongly in $C_w((0,T),L^\Gamma(\TT))$, we deduce that
$
\psi*\tilde \varrho_m \to \psi*\tilde\varrho
$
pointwisely. Further, since $\psi\in C^1(\TT)$ and $\pat \tilde\varrho_m\in L^q((0,T),L^q(\TT))$ we deduce a bound on $\nabla_{t,x}(\psi*\tilde\varrho_m)$. As $\TT$ is compact, we use the Arzel\`a-Ascoli theorem to deduce the demanded convergence. 

Next, directly from \eqref{eq.nonmom.conv} and \eqref{eq.mom.conv} we deduce that $\tilde\varrho,\tilde \vu$ satisfies
$$
\pat \tilde\varrho + \diver(\tilde\varrho\tilde\vu) = \varepsilon\Delta\tilde\varrho
$$
in the sense of distribution $\tilde{\mathbb P}$-almost surely. 

Next, it holds that
\begin{equation}\label{eq.nonlin.conv}
\tilde\varrho_m \tilde \vu_m\otimes \tilde \vu_m \to \tilde\varrho \tilde\vu \otimes\tilde\vu \ \mbox{weakly in }L^1((0,T), L^1(\mathbb T)).
\end{equation}
Indeed, since $L^{\frac{2\Gamma}{\Gamma+1}}(\mathbb T)$ is compactly embedded into $W^{-1,2}(\mathbb T)$, we get 
\begin{equation}\label{eq.strong.mom.conv}
\Pi_m(\tilde\varrho_m\tilde\vu_m)\to \tilde\varrho\tilde\vu\ \mbox{ in }L^2((0,T), W^{-1,2}(\mathbb T)).
\end{equation} $\tilde{\mathbb P}$-almost surely.
Further, bounds \eqref{eq:supremum}, \eqref{eq:rho2} allows to deduce
\begin{equation*}
\mathbb E\left(\|\tilde\vu_m\|_{L^2((0,T), W^{1,2}(\mathbb T))}\right)\leq c
\end{equation*}
similarly to \eqref{eq:odhad.rychlosti}. Therefore, we deduce the weak convergence $\tilde\vu_m\to \tilde\vu$ in $L^2((0,T), W^{1,2}(\mathbb T))$ $\tilde{\mathbb P}$-almost surely. This combined with \eqref{eq.strong.mom.conv} imply \eqref{eq.nonlin.conv}.

Next, we claim that
\begin{equation}\label{eq:convergenceF}
\Pi_m[F_{k,\varepsilon}(\tilde\varrho_m,\tilde\vu_m)]\to  F_{k,\varepsilon}(\tilde\varrho,\tilde\vu)\ \mbox{in }L^p(\tilde\Omega\times(0,T)\times \mathbb T)
\end{equation}
for any $p\in (1,\infty)$. Indeed, 
$$
\sqrt{\tilde\varrho_m} \tilde\vu_m \to \sqrt{\tilde\varrho}\tilde \vu
$$
in $L^2((0,T)\times \TT)$ $\tilde{\mathbb P}$-almost surely yielding $\sqrt{\tilde\varrho_m}\tilde\vu_m \to \sqrt{\tilde\varrho}\tilde\vu $ almost everywhere in $\tilde\Omega\times (0,T)\times\TT$. Next, the convergence \eqref{eq.nonmom.conv} yields $\tilde\varrho_m\to\tilde\varrho$ almost everywhere in $\tilde\Omega\times (0,T)\times \TT$. Next, the mapping
$$
(\varrho,\sqrt \varrho \vu)\mapsto F_{k,\varepsilon}\left(\varrho,\frac{\sqrt \varrho \vu}{\sqrt \varrho}\right) 
$$
is continuous and therefore
$$
F_{k,\varepsilon}(\tilde{\varrho}_m,\tilde\vu_m)\to F_{k,\varepsilon}(\tilde\varrho,\tilde\vu)
$$
almost everywhere. The bound \eqref{eq:bound.F} together with the Vitali convergence theorem then yield \eqref{eq:convergenceF}.

Next, the mass conservation law \eqref{eq:mass.cons} implies that 
$\tilde\varrho_m\in L^\infty(\tilde\Omega, L^1((0,T)\times\TT))$ and since $\tilde\varrho_m\in L^r(\tilde\Omega, L^\infty((0,T),L^\Gamma(\TT)))$ due to \eqref{eq:supremum}, we have (once again by the Vitali convergence theorem) that
\begin{equation}\label{eq:lim.m.stoch}
\tilde\varrho_m\to \tilde\varrho\ \mbox{strongly in } L^q(\tilde\Omega\times(0,T)\times\TT)
\end{equation}
for some $q>2$. This together with \eqref{eq:convergenceF} yield
$$
\tilde\varrho_m\Pi_m (F_{k,\varepsilon}(\tilde \varrho_m,\tilde\vu_m)\to \tilde\varrho F_{k,\varepsilon} (\tilde\varrho,\tilde\vu)\ \mbox{strongly in } L^2(\tilde\Omega\times(0,T)\times \TT).
$$
To pass to a limit in the stochastic integral we proceed similarly as in the previous section. At first, we strengthen \eqref{eq:lim.m.stoch} to obtain an appropriate convergence result for the stochastic process $\tilde\varrho_m \Pi_m [\mathbb F_\varepsilon(\tilde\varrho_m ,\tilde\vu_m)]$ and afterward, we apply Lemma~\ref{L266}.

From the properties of $\Pi_m$, inequalities \eqref{eq:bound.F}, \eqref{eq:supremum} we get
\begin{multline*}
\tilde{\mathbb{E}} \int_0^{T} \left\|\int_{\TT} \tilde\varrho_m \Pi_m [\mathbb F_\varepsilon(\tilde\varrho_m ,\tilde\vu_m)] {\rm d}x\right\|^2_{L_2(\mathfrak{U},\mathbb{R})} {\rm d}t = \tilde{\mathbb{E}} \int_0^{T} \sum_{k=1}^{\infty}\left(\int_{\TT} \tilde\varrho_m \Pi_m [F_{k,\varepsilon}(\tilde\varrho_m ,\tilde\vu_m)] {\rm d}x\right)^2 {\rm d}t\\
\leq c\sum_{k=1}^{\infty} f_{k,\varepsilon}^2 \tilde{\mathbb{E}}\int_0^{T}\left(\int_{\TT} \tilde\varrho^q_m {\rm d}x\right)^2{\rm d}t \leq c.
\end{multline*}
which leads to
\[
\tilde\varrho_m \Pi_m [F_{k,\varepsilon}(\tilde\varrho_m ,\tilde\vu_m)] \to \tilde\varrho  F_{k,\varepsilon}(\tilde\varrho ,\tilde\vu) \quad \mbox{in $L^r(\tilde\Omega,L^2((0,T),L^2(\TT)))$}
\]
for some $r>2$. It follows that
\[
\tilde\varrho_m \Pi_m [\mathbb F_\varepsilon(\tilde\varrho_m ,\tilde\vu_m)] \to \tilde\varrho  \mathbb{F}_\varepsilon(\tilde\varrho ,\tilde\vu) \quad \mbox{in $L^2(0,T,L_2(\mathfrak{U},L^2(\TT)))$ $\tilde{\mathbb{P}}$-a.s.}
\]
which together with the convergence of $\tilde W_m$ in \eqref{eq.nonmom.conv} and application of Lemma \ref{L266} gives a desired result.

As a consequence of the previous ideas we get 
\begin{multline}
\int_0^T\int_\TT \varrho \vu \cdot \pat\varphi \ {\rm d}x{\rm d}t  + \int_\TT \varrho_0\vu_0 \varphi(0,\cdot)\ {\rm d}x -\int_\TT \varrho(T,\cdot)\vu(T,\cdot) \varphi(T,\cdot)\ {\rm d}x\\
+ \int_0^T \int_\TT \varrho \vu\otimes\vu:\nabla \varphi + p_\delta(\varrho)\diver \varphi\ {\rm d}x{\rm d}t
 - \int_0^T\int_\TT S(D\vu): D\varphi  \ {\rm d}x{\rm d}t = \\
-\varepsilon\int_0^T\int_\TT \nabla (\varrho \vu) \nabla \varphi  \ {\rm d}x{\rm d}t +
\int_0^T\int_\TT \varrho \nabla K*\varrho  \varphi \ {\rm d}x{\rm d}t\\
 - \int_0^T\int_\TT\varrho(t,x) \varphi(x) \int_\TT \varrho(t,y) \psi(x-y) (\vu(t,y) - \vu(t,x)) \ {\rm d}x{\rm d}y{\rm d}t\\
+ \sum_{k=1}^\infty\int_0^T\int_\TT \vG_k(\varrho,\varrho\vu)\cdot\varphi \ {\rm d}x{\rm d}\tilde{W}_k\label{eq:mom.approx1}
\end{multline}
$\tilde{\mathbb P}$-almost surely for every test function $\varphi \in C_c^\infty([0,T]\times\mathbb T)$

Finally, we need to consider the limit in the energy balance. 
First we claim that
\begin{equation}\label{eq:104}
\int_0^T\int_{\mathbb T} \tilde\varrho_m \Pi_m [\mathbb F_\varepsilon(\tilde\varrho_m ,\tilde\vu_m)]\tilde\vu_m\ {\rm d}x{\rm d}\tilde W_m\to \int_0^T\int_{\mathbb T} \tilde\varrho \mathbb F_\varepsilon(\tilde\varrho,\tilde\vu)\tilde\vu \ {\rm d}x{\rm d}\tilde W
\end{equation}
in $L^2(0,T)$ in probability.

This can be deduced as in the previous subsection.
Taking into account \eqref{eq:104} and the lower-semi-continuity of convex functionals we obtain that the limits $\tilde\varrho,\ \tilde\vu$ satisfy \eqref{eq:ene.thirtyfour} without any projections $\Pi_m$.

\section{Vanishing of auxiliary terms}
We work with a dissipative martingale solution to \label{sec.vat}
\begin{equation}\label{eq:after.galerkin}
\begin{split}
{\rm d} \varrho + \diver (\varrho \vu)\ {\rm d}t  =& \varepsilon \Delta \varrho \ {\rm d}t\\
{\rm d} (\varrho \vu) + \diver (\varrho \vu \otimes \vu)\ {\rm d}t + \nabla p_\delta (\varrho)\ {\rm d}t =& \varepsilon \Delta(\varrho \vu)\ {\rm d}t + \diver S(D\vu)\ {\rm d}t\\ & - \varrho(\nabla K*\varrho)\ {\rm d}t + \varrho(\psi*(\varrho\vu) - \vu(\psi*\varrho))\ {\rm d}t\\
& + \varrho(\mathbb G_\varepsilon(\varrho,\vu))\ {\rm d}W
\end{split}
\end{equation}
and the goal of this section is to go with $\varepsilon$ and $\delta$ to $0$.

\subsection{Limit $\varepsilon \to 0$.}
Throughout this subsection, we consider $\delta>0$ to be fixed.

We assume that there is a probability $\Lambda$ on $L^1(\TT)\times L^1(\TT,\mathbb R^3)$ such that 
\begin{equation}\label{eq:Lambda.epsilon}
\Lambda\{\varrho>0\} =1,\ \Lambda\left\{\underline\varrho\leq \int_\TT \varrho\ {\rm d}x\leq \overline \varrho\right\} = 1\end{equation}
for some deterministic constants $0<\underline\varrho\leq \overline\varrho<\infty$, and
\begin{equation}\label{eq:Lambda.epsilon.2}
\int_{L^1(\TT)\times L^1(\TT,\mathbb R^3)}\left|\int_\TT \frac 12 \frac{|\vq|^2}\varrho+ P_\delta(\varrho)\right|^r\ {\rm d}\Lambda\leq c<\infty
\end{equation}
for some $r\geq 4$. 

Next, let $\varrho_0,\vq_0$ be an initial condition whose law is $\Lambda(\varrho_0,\vq_0)$. We define
$\varrho_{0,\varepsilon}\in C^{2+\nu}(\TT)$ such that $0<\varepsilon\leq \varrho_{0,\varepsilon}\leq \frac 1\varepsilon$, $\underline\varrho/2 \leq \int_\TT \varrho_{0,\varepsilon}\ {\rm d}x \leq 2\overline\varrho$ and
\begin{equation}\label{ini.p.e.1}
\varrho_{0,\varepsilon}\to \varrho_0\ \mbox{in }L^p(\Omega,L^\Gamma(\TT)),\ p\in [1,r\Gamma].
\end{equation}
Further, we define $\vq_{0,\varepsilon}$ such that $\frac{|\vq_{0,\varepsilon}|^2}{\varrho_{0,\varepsilon}}\in L^p(\Omega,L^1(\TT))$, $p\in [1,r]$  and
\begin{equation}
\begin{split}\label{ini.p.e.2}
\vq_{0,\varepsilon} & \to \vq_0\ \mbox{ in }L^p(\Omega;L^1(\TT)),\ p\in [1,r]\\
\frac{\vq_{0,\varepsilon}}{\sqrt{\varrho_{0,\varepsilon}}} & \to \frac{\vq_0}{\sqrt{\varrho_0}}\ \mbox{ in }L^p(\Omega,L^2(\TT)), \ p\in [1,2r].
\end{split}
\end{equation}

We denote by $\varrho_\varepsilon$ and $\vu_\varepsilon$ a solution to \eqref{eq:after.galerkin} associated to some coefficient $\varepsilon >0$. Without loss of generality we can assume that all the solutions are defined on the standard probability space $([0,1], \overline{\mathfrak{B}([0,1])}, \mathfrak{L})$ (it is a consequence of the Jakubowski-Skorokhod theorem). Moreover, we can assume there exists the common Wiener process $W$ for all $\varepsilon$ (because we can carry out the compactness argument from the previous section for any subsequence $\{\varepsilon_k\}_{k\in\mathbb{N}}$ at once).

Note that a solution to \eqref{eq:after.galerkin} is constructed in the previous section and we may assume that $\varrho_\varepsilon$ and $\vu_\varepsilon$ satisfy \eqref{eq:ene.thirtyfour}.
Similarly as before we deduce (also with the help of the mass conservation law)
\begin{equation}
\begin{split}
\forall \tau \in [0,T],\quad\int_{\mathbb T} \varrho(\tau,x)\ {\rm d}x &= \int_{\mathbb T} \varrho_0(x)\ {\rm d}x\\
\mathbb E\left[\left|\sup_{t\in [0,T]} \|\varrho\|_{L^\gamma(\TT)}^\gamma\right|^r\right] + \mathbb E \left[\|\sqrt \varepsilon \nabla \varrho\|_{L^2((0,T), L^2(\Omega))}^{2r}\right] &\leq c\\
\mathbb E \left[\left|\sup_{t\in[0,T]} \|\varrho |\vu|^2\|_{L^1(\mathbb T)}\right|^r + \left|\sup_{t\in[0,T]} \|\varrho \vu\|_{L^{2\Gamma/(\Gamma+1)}(\TT)}^{2\Gamma/(\Gamma+1)}\right|^r\right] &\leq c\\
\mathbb E \left[\|\nabla \vu\|_{L^2((0,T),L^2(\TT))}^{2r}\right] & \leq c
\end{split}\label{eq:odhady.18}
\end{equation}
where $c$ depends on the initial data. 

Next, we have
\begin{Lemma}
There is a constant $c$ such that\label{lem.roucha}
$$
\mathbb E\left[\left| \int_0^T \int_{\mathbb T} p(\varrho)\varrho + \delta \varrho^{\Gamma+1}\ {\rm d}x{\rm d}t\right|\right]\leq c.
$$
\end{Lemma}
\begin{proof}
We take a test function $\varphi = \nabla \Delta^{-1} (\varrho - (\varrho)_{\mathbb T})$ in \eqref{eq:mom.approx1} (see also \cite[Theorem A.4.1]{BrFeHo}).
We get
\begin{multline*}
\int_0^T\int_\mathbb T p_\delta(\varrho)  \varrho \ {\rm d}x {\rm d}t\\
 =\int_0^T\int_{\TT} \varrho \vu (\nabla\Delta^{-1}\diver (\varrho \vu))\ {\rm d}x{\rm d}t -\int_{\TT} \varrho_0\vu_0 \nabla \Delta^{-1}(\varrho_0 - (\varrho_0)_\TT)\ {\rm d}x\\
+ \int_\TT \varrho(T,\cdot)\vu(T,\cdot) \nabla \Delta^{-1}(\varrho(T,\cdot) - (\varrho(T,\cdot))_\TT)\ {\rm d}x\\
 - \int_0^T\int_\TT \varrho \vu \otimes\vu :\nabla^2 \Delta^{-1} (\varrho -(\varrho)_\TT)\ {\rm d}x{\rm d}t \\
 +\int_0^T \int_\TT p_\delta(\varrho)(\varrho)_\TT\ {\rm d}x{\rm d}t + \int_0^T\int_\TT S(D\vu): D(\nabla\Delta^{-1} (\varrho - (\varrho)_{\TT}))\ {\rm d}x{\rm d}t\\
 -\varepsilon\int_0^T \int_\TT\nabla(\varrho \vu)\nabla^2 \Delta^{-1} (\varrho - (\varrho)_\TT) \ {\rm d}x{\rm d}t + \int_0^T\int_\TT \varrho (\nabla K*\varrho)\nabla\Delta^{-1} (\varrho - (\varrho)_\TT)\ {\rm d}x{\rm  d}t\\
-\int_0^T\int_\TT \varrho \nabla\Delta^{-1} (\varrho - (\varrho)_\TT) (\psi*(\varrho \vu) - \vu (\psi*\varrho)) \ {\rm d}x{\rm d}t\\
+\sum_{k=1}^\infty \int_0^T \int_\TT \mathbb G_{k,\varepsilon}(\varrho,\varrho\vu) \cdot \nabla\Delta^{-1}(\varrho - (\varrho)_\TT)\ {\rm d}x{\rm d}W_k
 =: \sum_{i=1}^{10} \mathcal I_i.
\end{multline*}

Clearly, the boundedness of the right hand side yields the demanded claim. Note also that $\Delta^{-1}:W^{k,p}(\TT)\mapsto W^{k+2,p}(\TT)$ is a bounded linear operator for $k \in \{-1,0,1,\ldots\}$ and $p\in (1,\infty)$. 
In what follows, we extensively use the energetic estimates as well as the fact that they imply (with help of \cite[Lemma 3.2]{feireisl}) that
\begin{equation*}
\mathbb E \left[\|\vu\|^{2r}_{L^2((0,T), L^6(\TT))}\right]\leq c.
\end{equation*}
We have
\begin{equation*}
|\mathcal I_1|\leq c\int_0^T\int_\TT |\varrho \vu|^2 \ {\rm d}x{\rm d}t\leq \int_0^T\|\varrho\|_{L^4(\TT)}^2 \|\vu\|_{L^4}^2\leq \|\varrho\|_{L^\infty((0,T),L^\Gamma(\TT))}^2 \|\vu\|_{L^2((0,T), L^4(\TT))}^2.
\end{equation*}
The bound of $\mathcal I_2$ comes from the assumptions imposed on initial data. Next, since $\varrho\in L^\infty((0,T),L^\Gamma(\TT))$ implies $\nabla\Delta^{-1} (\varrho - (\varrho)_\TT)\in L^\infty((0,T),W^{1,\Gamma}(\TT))\subset L^\infty((0,T), L^\infty(\TT))$ we get by the H\"older inequality
\begin{multline*}
|\mathcal I_3| \leq \int_\TT \sqrt \varrho (\sqrt \varrho |\vu|) |\nabla \Delta^{-1} (\varrho - (\varrho)_\TT)|\ {\rm d}x \\
\leq \|\nabla\Delta^{-1}(\varrho - (\varrho)_\TT)\|_{L^\infty((0,T)\times \TT)} \|\varrho\|_{L^\infty((0,T),L^\Gamma(\TT))}^{\frac 12} \|\varrho |\vu|^2\|_{L^\infty((0,T), L^1(\TT))}^{\frac 12}.
\end{multline*}
Further,
\begin{equation*}
|\mathcal I_4| \leq c \|\varrho\|_{L^\infty((0,T),L^\Gamma(\TT))}^2  \|\vu\|_{L^2((0,T),L^6(\TT))}^2.
\end{equation*}
Recall that the mass conservation law implies $(\varrho)_\TT$ is constant over time and thus
\begin{equation*}
|\mathcal I_5| \leq c \int_0^T\int_\TT p_\delta(\varrho)\ {\rm d}x{\rm d}t\leq c \|\varrho\|_{L^\infty((0,T),L^\Gamma(\TT))}.
\end{equation*}
Next, the Sobolev embedding yields
\begin{multline*}
|\mathcal I_6| \leq c\|\vu\|_{L^2((0,T),W^{1,2}(\TT))} \|\Delta^{-1}(\varrho - (\varrho)_\TT)\|_{L^2((0,T),W^{1,2}(\TT))}\\
\leq  c\|\vu\|_{L^2((0,T),W^{1,2}(\TT))} \|\varrho\|_{L^\infty((0,T),L^\Gamma(\TT))}
\end{multline*}
We have 
\begin{multline*}
|\mathcal I_7| \leq c\varepsilon \int_0^T\int_\TT |\nabla(\varrho\vu)|\varrho\ {\rm d}x{\rm d}t\leq c\varepsilon \int_0^T\int_\TT |\nabla \varrho||\vu|\varrho \ {\rm d}x{\rm d}t + c\varepsilon \int_0^T\int_\TT  \varrho^2 |\nabla \vu|\ {\rm d}x{\rm d}t\\
\leq c\varepsilon\left( \|\nabla \varrho\|_{L^2((0,T)\times\TT)} \|\vu\|_{L^2((0,T),L^6(\TT))}\|\varrho\|_{L^\infty((0,T),L^\Gamma(\TT))}\right.\\\left. + \|\varrho\|_{L^\infty((0,T),L^4(\TT))}^2 \|\nabla\vu\|_{L^2((0,T),L^2(\TT))}^2\right).
\end{multline*}
and \eqref{eq:odhady.18} yields $\mathcal E[|I_7|]\leq c$.
Further, 
$$
|\mathcal I_8|\leq c(K,M) \int_0^T\int_\TT \varrho \nabla\Delta^{-1} (\varrho - (\varrho)_\TT)\ {\rm d}x{\rm d}t \leq c(K,M) \|\varrho\|_{L^\infty((0,T),L^2(\TT))}^2.
$$
Next,
$$
|\mathcal I_9|\leq c(\psi) \|\varrho\vu\|_{L^\infty((0,T),L^1(\TT))} \|\varrho\|_{L^2((0,T)\times \TT)}^2 + c(\psi,M) \|\varrho\|_{L^4((0,T)\times\TT)}^4 \|\vu\|^2_{L^2((0,T)\times \TT)}
$$
and thus also $\mathcal I_9$ can be controlled by estimates in \eqref{eq:odhady.18}. \\
In order to bound $\mathcal I_{10}$ we employ the Burkholder-Davis-Gundy inequality to deduce
\begin{equation*}
\mathbb E[|\mathcal I_{10}|^r]\leq \mathbb E\left[\int_0^T\sum_{k=1}^\infty \left|\int_\TT \varrho F_{k,\varepsilon} (\varrho,\vu) \nabla \Delta^{-1} (\varrho - (\varrho)_\TT)\ {\rm d}x\right|^2\ {\rm d}t\right]^{r/2}.
\end{equation*}
Next, \eqref{eq:bound.F1} yields $|F_{k,\varepsilon}(\varrho,\vu)|\leq f_k|1+\vu|$ and thus we have
\begin{equation*}
\left|\int_\TT \varrho F_{k,\varepsilon} \nabla \Delta^{-1} (\varrho - (\varrho)_\TT)\ {\rm d}x\right|\leq f_k\|\varrho\|_{L^\infty((0,T),L^\Gamma(\TT))}\int_\TT \varrho + |\varrho \vu|\ {\rm d}x
\end{equation*}
which yields the demanded bound of $\mathcal I_{10}$.
\end{proof}

We introduce a Young measure $\nu_\varepsilon: [0,T]\times \mathbb T \mapsto \mathbb P(\mathbb R\times \mathbb R^3 \times \mathbb R^{3\times 3})$ as
\begin{equation*}
\nu_{\varepsilon,t,x} = \delta_{\varrho_\varepsilon(t,x), \vu_\varepsilon(t,x), \nabla \vu_{\varepsilon}(t,x)}.
\end{equation*}

Next,
\begin{equation*}
\mathcal X := \mathcal X_{\varrho_0} \times \mathcal X_{\vq_0} \times \mathcal X_{\vq_0/\sqrt{\varrho_0}} \times \mathcal X_\varrho\times \mathcal X_{\varrho \vu} \times \mathcal X_\vu \times \mathcal X_W \times \mathcal X_E \times \mathcal X_\nu.
\end{equation*}
where
\begin{equation*}
\begin{split}
\mathcal X_{\varrho_0} & = L^\Gamma (\mathbb T)\\
\mathcal X_{\vq_0} & = L^1(\mathbb T)\\
\mathcal X_{\vq_0/\sqrt{\varrho_0}} & = L^2(\mathbb T)\\
\mathcal X_\varrho & = L^{\Gamma+1}((0,T)\times \mathbb T,w)\cap C_w ([0,T], L^\Gamma(\mathbb T))\\
\mathcal X_{\varrho \vu} & = C_w ([0,T], L^{2\Gamma/(\Gamma+1)}(\mathbb T)) \cap C([0,T], W^{-k,2}(\mathbb T)),\ k>\frac 52\\
\mathcal X_\vu & = L^2((0,T),W^{1,2},w)\\
\mathcal X_W &  = C([0,T],\mathfrak U)\\
\mathcal X_E & = L^\infty([0,T], \mathcal M_b(\mathbb T), w^*)\\
\mathcal X_\nu & = L^\infty_{w^*}((0,T)\times \mathbb T, w^*)
\end{split}
\end{equation*}
Let $E_\delta(\varrho_\varepsilon,\vu_\varepsilon)$ be defined as
$$
E_\delta(\varrho_\varepsilon,\vu_\varepsilon) = \frac 12 \varrho_\varepsilon |\vu_\varepsilon|^2 + P_\delta(\varrho_\varepsilon) + \frac 12 \varrho_\varepsilon K*\varrho_\varepsilon.
$$
We are going to prove that
$$
\left\{\mathcal L\left[\varrho_{0,\varepsilon}, \vq_{0,\varepsilon}, \frac{\vq_{0,\varepsilon}}{\sqrt{\varrho_{0,\varepsilon}}}, \varrho_\varepsilon, \varrho_\varepsilon \vu_\varepsilon, \vu_\varepsilon, W, E_\delta(\varrho_\varepsilon,\vu_\varepsilon), \nu_\varepsilon\right],\varepsilon\in (0,1)\right\}
$$
is tight on $\mathcal X$. This is going to be done in several steps.

First, the tightness of laws of $\varrho_{0,\varepsilon}$, $\vq_{0,\varepsilon}$ and $\frac{\vq_{0,\varepsilon}}{\varrho_{0,\varepsilon}}$ follows directly from the definition of initial conditions (see \eqref{ini.p.e.1} and \eqref{ini.p.e.2}). The tightness of $W$ is also clear. 

The tightness of $\mathcal L[\vu_\varepsilon]$ in $\mathcal X_\vu$ can be obtained similarly as in the previous chapter -- see \eqref{eq:u.tight}.  The same true is also for the tightness of $\mathcal L[\varrho_\varepsilon \vu_\varepsilon]$ in $\mathcal X_{\varrho \vu}$ and $\mathcal L[\varrho_\varepsilon]$ in $\mathcal X_\varrho$. 

Next, $\mathcal L[E_\delta(\varrho_\varepsilon,\vu_\varepsilon)]$ is tight in $\mathcal X_E$. This follows directly from the definition of $E_\delta $ and \eqref{eq:odhady.18}. 

In order to prove that $\{\mathcal L[\nu_\varepsilon],\varepsilon \in (0,1)\}$ is tight on $\mathcal X_\nu$ we use \cite[Corollary 2.8.6]{BrFeHo}. We define
$$
B_R:=\left\{\nu \in L^\infty_{w^*}((0,T)\times \TT, \mathbb P(\mathbb R^{13})),\ \int_0^T\int_\TT\int_{\mathbb R^{13}} \left(|\xi_1|^{\Gamma+1} +\sum_{i=1}^{13}|\xi_i|^2\right)\ {\rm d}\nu_{t,x}(\xi){\rm d}x{\rm d}t\leq R\right\}. 
$$
this set is relatively compact in $(L^\infty_{w^*}((0,T)\times \TT, \mathbb P(\mathbb R^{13})),w^*)$. Further, we deduce that
\begin{multline*}
\mathcal L[\nu_\varepsilon](B_R^c) = \mathbb P\left(\int_0^T\int_\TT\int_{\mathbb R^{13}} \left(|\xi_1|^{\Gamma+1} +\sum_{i=1}^{13}|\xi_i|^2\right)\ {\rm d}\nu_{t,x}(\xi){\rm d}x{\rm d}t > R\right)\\
 = \mathbb P \left(\int_0^T\int_\TT |\varrho_\varepsilon|^{\Gamma+1} + |\vu_\varepsilon|^2 + |\nabla\vu_\varepsilon|^2\ {\rm d}x{\rm d}t>R\right)\leq  \frac CR.
\end{multline*}
which yields the desired tightness.

Therefore, we use the Jakubowski-Skorokhod theorem in order to deduce the existence of 
$$
\left(\tilde\varrho_{0,\varepsilon}, \tilde\vq_{0,\varepsilon}, {\frac{\tilde\vq_{0,\varepsilon}}{\sqrt{\varrho_{0,\varepsilon}}}}, \tilde\varrho_\varepsilon, \widetilde{\varrho_\varepsilon \vu_\varepsilon}, \tilde\vu_\varepsilon, \tilde W_\varepsilon, \tilde E_\delta(\varrho_\varepsilon,\vu_\varepsilon), \tilde\nu_\varepsilon\right)
$$
and 
$$
\left(\tilde\varrho_{0}, \tilde\vq_{0}, {\frac{\tilde\vq_{0}}{\sqrt{\varrho_{0}}}}, \tilde\varrho, \widetilde{\varrho\vu}, \tilde\vu, \tilde W, \tilde E_\delta(\varrho,\vu), \tilde\nu\right)
$$
such that
\begin{itemize}
\item the laws of $$\left(\tilde\varrho_{0,\varepsilon}, \tilde\vq_{0,\varepsilon}, {\frac{\tilde\vq_{0,\varepsilon}}{\sqrt{\varrho_{0,\varepsilon}}}}, \tilde\varrho_\varepsilon, \widetilde{\varrho_\varepsilon \vu_\varepsilon}, \tilde\vu_\varepsilon, \tilde W_\varepsilon, \tilde E_\delta(\varrho_\varepsilon,\vu_\varepsilon), \tilde\nu_\varepsilon\right)$$ and  $$\left(\varrho_{0,\varepsilon}, \vq_{0,\varepsilon}, \frac{\vq_{0,\varepsilon}}{\sqrt{\varrho_{0,\varepsilon}}}, \varrho_\varepsilon, \varrho_\varepsilon \vu_\varepsilon, \vu_\varepsilon, W_\varepsilon, \tilde E_\delta(\varrho_\varepsilon,\vu_\varepsilon), \nu_\varepsilon\right)$$ coincide on $\mathcal X$.
\item The law of $\left(\tilde\varrho_{0}, \tilde\vq_{0}, \tilde{\frac{\vq_{0}}{\sqrt{\varrho_{0}}}}, \tilde\varrho, \widetilde{\varrho\vu}, \tilde\vu, \tilde W, \tilde E_\delta(\varrho,\vu), \tilde\nu\right)$ on $\mathcal X$ is a Radon measure.
\item The sequence $$
\left(\tilde\varrho_{0,\varepsilon}, \tilde\vq_{0,\varepsilon}, {\frac{\tilde\vq_{0,\varepsilon}}{\sqrt{\varrho_{0,\varepsilon}}}}, \tilde\varrho_\varepsilon, \widetilde{\varrho_\varepsilon \vu_\varepsilon}, \tilde\vu_\varepsilon, \tilde W_\varepsilon, \tilde E_\delta(\varrho_\varepsilon,\vu_\varepsilon), \tilde\nu_\varepsilon\right)
$$ 
converge to $$\left(\tilde\varrho_{0}, \tilde\vq_{0}, {\frac{\tilde\vq_{0}}{\sqrt{\varrho_{0}}}}, \tilde\varrho, \widetilde{\varrho\vu}, \tilde\vu, \tilde W, \tilde E_\delta(\varrho,\vu), \tilde\nu\right)$$ in $\mathcal X$ almost surely.
\end{itemize}
We obtained that the random distribution $[\tilde \varrho, \tilde \vu, \tilde W]$ satisfies
\begin{equation*}
{\rm d}\tilde\varrho + \diver \tilde\varrho\tilde \vu {\rm d}t = 0
\end{equation*}
in the sense of distributions 
and\footnote{Here we adopt the following notation. For a general function $f(\varrho_\varepsilon,\vu_\varepsilon)$ we use $\overline{f(\varrho,\vu)}$ to denote its weak limit (which is generaly not equal to $f(\varrho,\vu)$).}
\begin{multline}\label{limita.po.epsilon}
\int_0^T \int_\TT \tilde\varrho \tilde\vu \partial_t\varphi \ {\rm d}x{\rm d}t + \int_\TT \tilde \varrho_0 \tilde\vu_0 \varphi(0,\cdot)\ {\rm d}x - \int_\TT \tilde\varrho (T,\cdot)\tilde\vu(T,\cdot)\varphi(T,\cdot)\ {\rm d}x\\
 + \int_0^T\int_\TT \tilde\varrho \vu \otimes \tilde\vu :\nabla \varphi + \overline{p_\delta(\tilde\varrho)}\diver \varphi\ {\rm d}x{\rm d}t
 - \int_0^T\int_\TT S(D\overline u): D\varphi \ {\rm d}x{\rm d}t = 
\int_0^T\int_\TT \tilde\varrho \nabla K*\tilde\varrho \varphi \ {\rm d}x{\rm d}t\\
-\int_0^T\int_\TT \tilde\varrho \varphi (\psi*(\tilde\varrho\tilde\vu)-\tilde\vu(\psi*\tilde\varrho))\ {\rm d}x{\rm d}t\\
+ \sum_{k=1}^\infty \int_0^T\int_\TT \overline{\vG_k (\tilde\varrho,\tilde\varrho\tilde\vu)}\cdot \varphi \ {\rm d}x{\rm d}W.
\end{multline}
This follows as soon as we prove that 
\begin{equation}\label{eq:F.convergence}
\tilde \varrho_\varepsilon F_{k,\varepsilon}(\tilde\varrho_\varepsilon,\tilde\vu_\varepsilon)\to \overline{\tilde \varrho F_k (\tilde\varrho,\tilde\vu)} = \overline{G_k(\tilde\varrho,\tilde\varrho\tilde\vu)}
\end{equation}
in $L^2((0,T), W^{-l,2}(\TT))$. First, we have
\begin{multline*}
\int_\TT \tilde\varrho_\varepsilon |F_{k,\varepsilon}(\tilde\varrho_\varepsilon, \tilde \vu_\varepsilon) - F_k(\tilde \varrho_\varepsilon,\tilde\vu_\varepsilon)|\ {\rm d}x \leq \int_{\tilde\varrho_\varepsilon<\varepsilon} \tilde\varrho_\varepsilon |F_k(\tilde\varrho_\varepsilon,\tilde \vu_\varepsilon)|\ {\rm d}x + \int_{|\tilde\vu_\varepsilon|>\frac 1\varepsilon} \tilde \varrho_\varepsilon |F_k(\tilde\varrho_\varepsilon,\tilde\vu_\varepsilon)|\ {\rm d}x\\
 \leq \varepsilon \int_\TT |F_k(\tilde\varrho_\varepsilon,\tilde\vu_\varepsilon)|\ {\rm d}x +\varepsilon \int_\TT \tilde\varrho_\varepsilon |\tilde\vu_\varepsilon| |F_k(\tilde\varrho_\varepsilon,\tilde \vu_\varepsilon)|\ {\rm d}x \leq \varepsilon f_k \int_\TT 1 + |\tilde\vu|_\varepsilon + \tilde \varrho_\varepsilon  |\tilde \vu|_\varepsilon + \tilde \varrho_\varepsilon |\tilde\vu_\varepsilon|^2\ {\rm d}x.
\end{multline*}
Next,
$$
\int_\TT \tilde \varrho_\varepsilon (F_k(\tilde\varrho_\varepsilon,\tilde \vu_\varepsilon) - F_k(\tilde\varrho_\varepsilon, \tilde \vu))\varphi\ {\rm d}x\leq c f_k \|\varphi\|_{L^\infty(\TT)} \|\sqrt{\tilde\varrho_\varepsilon}\|_{L^2(\TT)}\|\tilde\varrho_\varepsilon |\tilde \vu_\varepsilon - \tilde\vu|^2 \|_{L^1(\TT)}^{\frac 12}
$$
and the right hand side tends to zero due to the convergencies obtained by use of the Jakubowski-Skorokhod theorem. To show that 
$$
\tilde \varrho_\varepsilon F_k(\tilde \varrho_\varepsilon,\tilde \vu)\to \overline{\tilde \varrho F_k(\tilde\varrho,\tilde\vu)}
$$
we use the renormalized equation of continuity
\begin{equation*}
\pat b(\tilde\varrho_\varepsilon) + \diver (b(\tilde\varrho_\varepsilon) \tilde \vu_\varepsilon) + (b'(\tilde\varrho_\varepsilon)\tilde\varrho_\varepsilon - b(\tilde\varrho_\varepsilon))\diver \tilde \vu_\varepsilon = \varepsilon \diver(b'(\tilde\varrho_\varepsilon)\nabla\tilde\varrho_\varepsilon) - \varepsilon b''(\tilde\varrho_\varepsilon) |\nabla\tilde\varrho_\varepsilon|^2
\end{equation*}
to obtain a bound on $\pat b(\tilde\varrho_\varepsilon)$ and, therefore, also
$$
b(\tilde\varrho_\varepsilon)\to \overline b(\tilde \varrho)\ \mbox{in }C([0,T],L^2(\TT),w).
$$
Thus, for any $b$ and $B$ which are globally Lipschitz we obtain
$$
b(\tilde\varrho_\varepsilon) B(\tilde \vu) \to \overline{b(\tilde\varrho)B(\tilde\vu)}\ {\mbox{in }} L^2((0,T), W^{-l,2}(\TT)).
$$
Now it is enough to approximate $\varrho F_k(\varrho,\vu)$ by finite sums $\sum_i b_i(\varrho)B_i(\vu)$ and \eqref{eq:F.convergence} follows.

It remains to prove that $\overline{p_\delta(\tilde \varrho)} = p_\delta (\tilde \varrho)$ and $\overline{\tilde \varrho \mathbb \vG(\tilde \varrho,  \tilde \vu)} = \tilde \varrho \mathbb \vG(\tilde \varrho,  \tilde \vu)$. 
In order to do so, we use the method of effective viscous flux to prove that
\begin{equation}\label{eq:outcome.of.evf}
\int_\TT \overline{\tilde\varrho \log(\tilde\varrho)}(t,\cdot) - \tilde\varrho \log(\tilde\varrho)(t,\cdot)\ {\rm d}x \leq 0
\end{equation}
holds for all $t\in [0,T]$. Since the mapping $\tilde \varrho \mapsto \tilde\varrho \log (\tilde\varrho)$ is convex, we obtain
\begin{equation*}
\tilde\varrho_\varepsilon \to \tilde\varrho \quad \mbox{in } L^1((0,T), L^1(\TT)).
\end{equation*}
This is sufficient to get 
\begin{equation*}
\overline{p_\delta(\tilde\varrho)}= p_\delta(\tilde\varrho)\quad \mbox{and}\quad \overline{\tilde\varrho \vG(\tilde\varrho,\tilde\vu)} = \tilde\varrho  \vG(\tilde\varrho,\tilde\vu).
\end{equation*}

Let focus on \eqref{eq:outcome.of.evf}. First of all, we use $\nabla\Delta^{-1}(\tilde\varrho - (\tilde\varrho)_\TT)$ as a test function in \eqref{limita.po.epsilon} and we use $\nabla\Delta^{-1}(\tilde\varrho_\varepsilon - (\tilde\varrho_\varepsilon)_\TT)$ as a test function in \eqref{eq:mom.approx1}. We get
\begin{multline}
\int_0^T\int_{\mathbb T} \overline{p_\delta(\tilde\varrho)} \tilde\varrho - (\lambda + 2\mu)\diver \tilde\vu \tilde\varrho \ {\rm d}t{\rm d}t = \int_0^T\int_{\mathbb T} \overline{p_\delta(\tilde\varrho)} (\tilde\varrho)_{\mathbb T} \ {\rm d}x{\rm d}t\\
 - \int_0^T\int_\TT (\lambda + 2\mu) \diver \tilde\vu (\tilde\varrho)_\TT\ {\rm d}x{\rm d}t + \int_0^T\int_\TT \tilde\varrho\nabla K*\tilde\varrho \nabla\Delta^{-1}(\tilde\varrho - (\tilde\varrho)_\TT ) \ {\rm d}x{\rm d}t\\
-\int_0^T\int_\TT\tilde\varrho\nabla\Delta^{-1}(\tilde\varrho - (\tilde\varrho)_\TT) \left(\psi *(\tilde\varrho\tilde\vu) - \tilde\vu(\psi*\tilde\varrho)\right)\ {\rm d}x{\rm d}t\\
+\int_\TT \tilde\varrho (T,\cdot)\tilde\vu(T,\cdot) \nabla \Delta^{-1}\left(\tilde\varrho - (\tilde\varrho)_\TT\right)\ {\rm d}x - \int_\TT \tilde\varrho_0 \tilde\vu_0 \nabla\Delta^{-1} \left(\tilde\varrho_0 - (\tilde\varrho_0)_\TT\right)\ {\rm d}x\\
 + \sum_{k=1}^\infty \int_0^T\int_\TT \overline{\vG_k(\tilde\varrho,\tilde\varrho\tilde\vu)} \nabla\Delta^{-1}\left(\tilde\varrho - (\tilde\varrho)_\TT\right)\ {\rm d}x{\rm d}W\\
+ \int_0^T\int_\TT \tilde\varrho \tilde \vu_i \left(\partial_i \Delta^{-1}(\partial_j\left(\tilde\varrho\tilde\vu_j\right))   -\tilde\vu_j\partial_i \partial_j \Delta^{-1}(\tilde\varrho - (\tilde\varrho)_\TT)\right)\ {\rm d}x{\rm d}t\label{test.po.e}
\end{multline}
and
\begin{multline}
\int_0^T\int_{\mathbb T} \overline{p_\delta(\tilde\varrho_\varepsilon)} \tilde\varrho_\varepsilon - (\lambda + 2\mu)\diver \tilde\vu_\varepsilon \tilde\varrho_\varepsilon \ {\rm d}t{\rm d}t = \int_0^T\int_{\mathbb T} \overline{p_\delta(\tilde\varrho_\varepsilon)} (\tilde\varrho_\varepsilon)_{\mathbb T} \ {\rm d}x{\rm d}t\\
 - \int_0^T\int_\TT (\lambda + 2\mu) \diver \tilde\vu_\varepsilon (\tilde\varrho_\varepsilon)_\TT\ {\rm d}x{\rm d}t + \int_0^T\int_\TT \tilde\varrho_\varepsilon\nabla K*\tilde\varrho_\varepsilon \nabla\Delta^{-1}(\tilde\varrho_\varepsilon - (\tilde\varrho_\varepsilon)_\TT ) \ {\rm d}x{\rm d}t\\
-\int_0^T\int_\TT\tilde\varrho_\varepsilon\nabla\Delta^{-1}(\tilde\varrho_\varepsilon - (\tilde\varrho_\varepsilon)_\TT) \left(\psi *(\tilde\varrho_\varepsilon\tilde\vu_\varepsilon) - \tilde\vu_\varepsilon(\psi*\tilde\varrho_\varepsilon)\right)\ {\rm d}x{\rm d}t\\
+\int_\TT \tilde\varrho_\varepsilon (T,\cdot)\tilde\vu_\varepsilon(T,\cdot) \nabla \Delta^{-1}\left(\tilde\varrho_\varepsilon - (\tilde\varrho_\varepsilon)_\TT\right)\ {\rm d}x - \int_\TT \tilde\varrho_0 \tilde\vu_0 \nabla\Delta^{-1} \left(\tilde\varrho_0 - (\tilde\varrho_0)_\TT\right)\ {\rm d}x\\
 + \sum_{k=1}^\infty \int_0^T\int_\TT \overline{\vG_k(\tilde\varrho_\varepsilon,\tilde\varrho_\varepsilon\tilde\vu_\varepsilon)} \nabla\Delta^{-1}\left(\tilde\varrho_\varepsilon - (\tilde\varrho_\varepsilon)_\TT\right)\ {\rm d}x{\rm d}W\\
+ \int_0^T\int_\TT \tilde\varrho_\varepsilon \tilde \vu_{\varepsilon i} \left(\partial_i \Delta^{-1}(\partial_j\left(\tilde\varrho_\varepsilon\tilde\vu_{\varepsilon j}\right))   -\tilde\vu_{\varepsilon j}\partial_i \partial_j \Delta^{-1}(\tilde\varrho_\varepsilon - (\tilde\varrho_\varepsilon)_\TT)\right)\ {\rm d}x{\rm d}t\\
-\varepsilon\int_0^T\int_\TT \nabla(\tilde\varrho_\varepsilon\tilde\vu_\varepsilon)\nabla^2\Delta^{-1}(\tilde\varrho_\varepsilon - (\tilde\varrho_\varepsilon)_\TT)\ {\rm d}x{\rm d}t + \varepsilon \int_0^T\int_\TT \tilde\varrho_\varepsilon\tilde\vu_\varepsilon\nabla\tilde\varrho_\varepsilon\ {\rm d}x{\rm d}t\label{test.pred.e}
\end{multline}
Next, we have
\begin{equation*}
\nabla\Delta^{-1}(\tilde\varrho_\varepsilon- (\tilde\varrho_\varepsilon)_\TT)\to \nabla\Delta^{-1}(\tilde\varrho - (\tilde\varrho)_\TT) \ \mbox{ in }L^\infty((0,T),L^q(\Omega))\ \mbox{for all } q\in (1,\infty)
\end{equation*}
Due to \cite[Lemma 3.4]{FeNoPe}  we infer
$$
\tilde\varrho_\varepsilon \partial_i \Delta^{-1} \partial_j (\tilde\varrho_\varepsilon \tilde vu_{\varepsilon j}) - \tilde\vu_{\varepsilon j} \partial_i\partial_j \Delta^{-1}(\tilde\varrho_\varepsilon) \to \tilde\varrho \partial_i (\Delta^{-1} \partial_j (\tilde\varrho \tilde\vu_j)) - \tilde\vu_j \partial_i\partial_j \Delta^{-1}(\tilde\varrho)
$$
strongly in $L^\infty_{weak} ((0,T),L^{2\Gamma/(\Gamma+3)}(\TT))$. Consequently, every term in \eqref{test.pred.e} converge to its counterpart in \eqref{test.po.e} (with help of convergences deduced by Jakubowski-Skorokhod theorem) and we get
\begin{equation}\label{eq:evf}
\lim_{\varepsilon \to 0} \int_0^T\int_\TT p_\delta(\tilde\varrho_\varepsilon)\tilde\varrho_\varepsilon - (\lambda+2\mu) \diver \tilde\vu_\varepsilon \tilde\varrho_\varepsilon\ {\rm d}x{\rm d}t = \int_0^T\int_\TT \overline{p_\delta(\tilde\varrho)}\tilde\varrho - (\lambda+2\mu)\diver\tilde\vu \tilde\varrho\ {\rm d}x{\rm d}t
\end{equation}
Recall that $\tilde\varrho_\varepsilon,\ \tilde\vu_\varepsilon$ are regular enough  to  fulfill 
\begin{equation*}
\pat b(\tilde\varrho_\varepsilon) + \diver (b(\tilde\varrho_\varepsilon)\tilde\vu_\varepsilon) +  (b'(\tilde\varrho_\varepsilon)\tilde\varrho_\varepsilon - b(\tilde\varrho_\varepsilon))\diver \tilde\vu_\varepsilon \leq \varepsilon \Delta b(\tilde\varrho_\varepsilon)
\end{equation*}
for every $b$ convex and globally Lipschitz. Moreover, the renormalized continuity equation is true also for $\tilde\varrho$ and $\tilde\vu$ in the following form
\begin{equation*}
\pat b(\tilde\varrho) + \diver(b(\tilde\varrho)\tilde\vu) + (b'(\tilde\varrho)\tilde\varrho - b(\tilde\varrho))\diver\tilde\vu = 0.
\end{equation*}
We take $b(\varrho) = \varrho \log(\varrho)$ in the renormalized continuity equations and we integrate them over $(0,t)\times\TT$ to get
\begin{equation*}
\int_0^t\int_\TT \tilde\varrho_\varepsilon \diver \tilde\vu \ {\rm d}x{\rm d}t \leq \int_\TT\varrho_0 \log \varrho_0 \ {\rm d}x - \int_\TT \tilde\varrho_\varepsilon(t,\cdot) \log \tilde\varrho_\varepsilon(t,\cdot)\ {\rm d}x
\end{equation*}
and
\begin{equation*}
\int_0^t\int_\TT \tilde\varrho\diver\tilde\vu \ {\rm d}x{\rm d}t  = \int_\TT \varrho_0 \log\varrho_0 \ {\rm d}x - \int_\TT \tilde\varrho(t,\cdot) \log \tilde\varrho(t,\cdot)\ {\rm d} x.
\end{equation*}
We subtract these two equations and we send $\varepsilon$ to zero to deduce
\begin{equation*}
\int_\TT \overline{\tilde\varrho\log (\tilde\varrho)} - \tilde\varrho\log\tilde\varrho)(t,\cdot) \ {\rm d}x + \int_0^t\int_\TT (\overline{\tilde\varrho\diver\tilde\vu} - \tilde\varrho\diver\tilde\vu)\ {\rm  d}x{\rm d}t\leq 0.
\end{equation*}
The monotonicity of pressure yields
\begin{equation*}
\lim_{\varepsilon\to 0} \int_0^t \int_\TT p_\delta(\tilde\varrho_\varepsilon)\tilde\varrho_\varepsilon \ {\rm d}x{\rm d}t \geq \int_0^t \int_\TT \overline{p_\delta(\tilde\varrho)}\tilde\varrho\ {\rm d}x{\rm d}t
\end{equation*}
and we deduce with help of \eqref{eq:evf} that
\begin{equation*}
\int_0^t \int_\TT (\overline{\tilde\varrho\diver\tilde\vu} - \tilde\varrho\diver\tilde\vu)\ {\rm d}x{\rm d}t\geq 0
\end{equation*}
and \eqref{eq:outcome.of.evf} follows.

In order to pass to a limit in the energy inequality, we have to show that
\begin{equation*}
\int_0^t\int_\TT \tilde\varrho_\varepsilon \mathbb F_\varepsilon(\tilde\varrho_\varepsilon,\tilde\vu_\varepsilon)\tilde\vu_\varepsilon\ {\rm d}x{\rm d}\tilde W_\varepsilon\to \int_0^t\int_\TT \mathbb G(\tilde\varrho,\tilde\vu) \ {\rm d}x{\rm d}\tilde W
\end{equation*}
However, the proof of this inequality does not differ from the proof of \cite[Proposition 4.4.13]{BrFeHo}.

Thus, we have just proven that for any $\delta>0$ there exists a solution satisfying the continuity equation
\begin{equation*}
{\rm d}\varrho + \diver \varrho \vu {\rm d}t = 0
\end{equation*}
in a renormalized sense,
the momentum equation
\begin{multline*}
\int_0^T\int_\TT \varrho \vu \cdot \pat\varphi \ {\rm d}x{\rm d}t  + \int_\TT \varrho_0\vu_0 \varphi(0,\cdot)\ {\rm d}x -\int_\TT \varrho(T,\cdot)\vu(T,\cdot) \varphi(T,\cdot)\ {\rm d}x\\
+ \int_0^T \int_\TT \varrho \vu\otimes\vu:\nabla \varphi + p_\delta(\varrho)\diver \varphi\ {\rm d}x{\rm d}t
 - \int_0^T\int_\TT S(D\vu): D\varphi  \ {\rm d}x{\rm d}t = \\
 +
\int_0^T\int_\TT \varrho \nabla K*\varrho  \varphi \ {\rm d}x{\rm d}t\\
 - \int_0^T\int_\TT\varrho(t,x) \varphi(x) \int_\TT \varrho(t,y) \psi(x-y) (\vu(t,y) - \vu(t,x)) \ {\rm d}x{\rm d}y{\rm d}t\\
+ \sum_{k=1}^\infty\int_0^T\int_\TT \vG_k(\varrho,\varrho\vu)\cdot\varphi \ {\rm d}x{\rm d}W_k
\end{multline*}
$\mathbb P$ almost surely for all test function $\varphi\in C^\infty_c([0,T]\times\TT)$,
and 
the energy inequality
\begin{multline}\label{eq:ene.ine.po.eps}
\frac 12 \int_{\mathbb T} \varrho |\vu|^2(T,\cdot) \Phi(T)\ {\rm d}x  - \frac 12 \int_{\mathbb T}\varrho |\vu|^2(0,\cdot) \Phi(0)\ {\rm d}x + \int_0^T \int_{\mathbb T} S(D\vu):D\vu  \Phi\ {\rm d}x{\rm d}t\\
 + \varepsilon \int_0^T \int_{\mathbb T} \varrho |\nabla \vu|^2 \Phi \ {\rm d}x{\rm d}t + \frac 12 \int_{\mathbb T} \varrho K*\varrho(T,\cdot) \Phi(T)\ {\rm d}x - \frac 12 \int_{\mathbb T} \varrho K*\varrho (0,\cdot) \Phi(0)\ {\rm d}x \\ 
 + \int_0^T \int_{\mathbb T} \int_{\mathbb T} \varrho(t,x) \psi(x-y) \varrho(t,y) (\vu(t,y) - \vu(t,x))^2\Phi(t) \ {\rm d}x{\rm d}y {\rm d}t\\
 + \int_{\mathbb T} P_\delta (\varrho)(T,\cdot) \Phi(T) \ {\rm d} x - \int_{\mathbb T} P_{\delta}(\varrho)(0,\cdot)\ {\rm d}x \leq \frac 12 \int_0^T \int_{\mathbb T} \varrho K*\varrho \partial_t \Phi \ {\rm d}x {\rm d}t \\
+\int_0^T\int_{\mathbb T}(K*\varrho)\diver (\varrho (\vu - [\vu]_R))\Phi\ {\rm d}x{\rm d}t + \frac 12 \int_0^T\int_{\mathbb T} \varrho \vu^2 \partial_t \Phi \ {\rm d}x{\rm d}t\\
 + \int_0^T\int_{\mathbb T} P_\delta(\varrho) \partial_t\Phi\ {\rm d}x{\rm d}t
 + \int_0^T \int_{\mathbb T}  \vu \mathbb G(\varrho,\vu) \Phi \ {\rm d}x{\rm d}W + \frac 12 \int_0^T \int_{\mathbb T} \varrho^{-1} (\mathbb G(\varrho,\vu))^2\Phi \ {\rm d}x{\rm d}t
\end{multline}
for all $\Phi\in C^\infty([0,T])$.

\subsection{Limit $\delta \to 0$}
First, we construct initial conditions in such a way that $\varrho_{0,\delta}$ and $\vq_{0,\delta}$ are given by the law mentioned in the previous subsection. In particular, let 
$\Lambda$ be a probability measure satisfying assumptions of the main theorem. 
 We take $\varrho_0$ and $\vq_0$ in such a way that their law is $\Lambda$. Next, we choose $\varrho_{0,\delta}\in L^\Gamma(\TT)$ $\mathbb P$-almost surely  and $\vq_{0,\delta}$ such that for every $\delta>0$ the corresponding law satisfies \eqref{eq:Lambda.epsilon} and \eqref{eq:Lambda.epsilon.2} and the following convergencies are true for $\delta \to 0$:
\begin{equation*}
\begin{split}
&\varrho_{0,\delta}\to \varrho_0\ \mbox{in }L^\gamma(\TT)\ \mathbb P\mbox{-almost surely}\\
&\frac{\vq_{0,\delta}}{\sqrt{\varrho_{0,\delta}}}\to \frac{\vq_0}{\sqrt{\varrho_0}}\ \mbox{in }L^2(\TT)\ \mathbb P\mbox{-almost surely}\\
&\int_\TT \frac12 \frac{|\vq_{0,\delta}|^2}{\varrho_{0,\delta}} + P_\delta(\varrho_{0,\delta})\ {\rm d}x\to \int_\TT \frac12 \frac{|\vq_{0}|^2}{\varrho_{0}} + P(\varrho_0)\ {\rm d}x
\end{split}
\end{equation*}

Let $(\varrho_\delta,\vu_\delta)$ be a solution as presented in the previous subsection. Our aim is to send $\delta$ to zero in order to obtain a dissipative martingale solution to the system introduced in the beginning of this paper. 

The energy inequality \eqref{eq:ene.ine.po.eps} yields
\begin{equation}
\begin{split}\label{ene.est.for.delta}
\mathbb E\left(\left|\sup_{t\in [0,T]} \|\varrho_\delta\|_{L^\gamma(\TT)}^\gamma\right|^r\right)&\leq c\\
\mathbb E\left(\left|\sup_{t\in [0,T]} \delta\|\varrho_\delta\|_{L^\Gamma(\TT)}^\Gamma\right|^r\right) & \leq c\\
\mathbb E\left(\left| \sup_{t\in[0,T]}\|\varrho_\delta|\vu_\delta|^2\|_{L^1(\TT)}\right|^r\right) & \leq c\\
\mathbb E\left(\left| \sup_{t\in [0,T]} \|\varrho_\delta \vu_\delta\|_{L^{2\gamma/(\gamma+1)}(\TT)}^{2\gamma/(\gamma+1)}\right|^r\right) & \leq c\\
\mathbb E\left(\|\vu\|_{L^2((0,T),W^{1,2}(\TT))}^{2r}\right) & \leq c
\end{split}
\end{equation}
where $c\in \mathbb R$ is independent of $\delta$. Further, we have
\begin{Lemma}
There is $\beta>0$ such that
\begin{equation*}
\mathbb E\left(\left|\int_0^T\int_\TT (p(\varrho_\delta) + \delta(\varrho_\delta^\Gamma+\varrho_\delta^2))\varrho_\delta^\beta\ {\rm d}x{\rm d}t\right|\right)\leq c
\end{equation*}
with $c$ independent of $\delta$. 
\end{Lemma}
\begin{proof}
It suffices to test the momentum equation by
$
\nabla\Delta^{-1}(\varrho^\beta - (\varrho^\beta)_\TT)
$
where $\beta>0$ is a suitable parameter determined later. 
We obtain
\begin{multline*}
\int_0^T\int_\TT p_\delta(\varrho_\delta) \varrho_\delta^\beta \ {\rm d}x{\rm d}t = \int_0^T\int_\TT p_\delta(\varrho_\delta) (\varrho_\delta^\beta)_\TT \ {\rm d}x{\rm d}t \\
+ \int_\TT \varrho_\delta(T,\cdot)\vu_\delta(T,\cdot) \nabla\Delta^{-1} (\varrho_\delta^\beta - (\varrho_\delta^\beta)_\TT)\ {\rm d}x\\
 + \int_0^T\int_\TT S(D\vu_\delta): \nabla^2\Delta^{-1} (\varrho^\beta_\delta - (\varrho_\delta^\beta)_\TT)\ {\rm d}x{\rm d}t - \int_0^T\int_\TT \varrho_\delta\vu_\delta \nabla\Delta^{-1} (\pat \varrho_\delta^\beta)\ {\rm d}x{\rm d}t\\
- \int_\TT\varrho_0\vu_0 \nabla\Delta^{-1}(\varrho_0^\beta - (\varrho_0^\beta)_\TT)\ {\rm d}x{\rm d}t + \int_0^T\int_\TT \varrho_\delta \nabla K*\varrho_\delta \nabla \Delta^{-1} (\varrho_\delta^\beta - (\varrho_\delta^\beta)_\TT)\ {\rm d}x{\rm d}t\\
 - \int_0^T\int_\TT \varrho_\delta \nabla\Delta^{-1} (\varrho^\beta_\delta - (\varrho^\beta_\delta)_\TT) \left(\psi*(\varrho_\delta\vu_\delta) - \vu_\delta (\psi*\varrho_\delta)\right)\ {\rm d}x{\rm d}t\\
+\sum_{k=1}^\infty \int_0^T\int_\TT \vG_k(\varrho_\delta,\varrho_\delta\vu_\delta) \nabla\Delta^{-1} (\varrho_\delta^\beta - (\varrho_\delta^\beta)_\TT)\ {\rm d}x{\rm d}W_k\\
 - \int_0^T\int_\TT \varrho_\delta\vu_\delta\otimes\vu_\delta :\nabla^2\Delta^{-1} (\varrho_\delta^\beta - (\varrho_\delta^\beta)_\TT)\ {\rm d}x{\rm d}t
\end{multline*}
First, let focus on the fourth term on the right hand side. We deduce from the continuity equation that
$$
\pat (\varrho_\delta^\beta) = -\diver (\varrho_\delta^\beta \vu_\delta)  + (1-\beta) \varrho^\beta_\delta \diver \vu_\delta.
$$
Recall also that $\Delta^{-1}: W^{k,p}(\TT)\mapsto W^{k+2,p}(\TT)$ for all $k\in \{-1,0,1,\ldots\}$ and $p\in (1,\infty)$.
We deduce from \eqref{ene.est.for.delta} $\mathbb E \left(\|\varrho_\delta^\beta \diver \vu_\delta\|^r_{L^2((0,T),L^{2\gamma/(\gamma + 2\beta)}(\TT))}\right)\leq c$. We may also infer that $\mathbb E \left(\|\varrho_\delta\vu_\delta\|^r_{L^\infty((0,T),L^{2\gamma/(\gamma+1)}(\TT))}\right)\leq c$ and $\mathbb E \left(\|\varrho_\delta^\beta \vu_\delta\|_{L^2((0,T),L^{6\gamma/(\gamma+6\beta)}(\TT))}^r\right)\leq c$. Hence, we deduce that the fourth term is bounded assuming $\beta< \frac 23 \gamma - 1$. The remaining terms can be bound by use of \eqref{ene.est.for.delta} (compare also with the previous section and the proof of Lemma \ref{lem.roucha}).
\end{proof}

Now we are ready to use the Jakubowski-Skorokhod theorem. In particular, the set of laws
$$
\left\{\mathcal L\left[\varrho_{0,\delta},\vq_{0,\delta}, \frac{\vq_{0,\delta}}{\sqrt{\varrho_{0,\delta}}}, \varrho_\delta,\varrho_\delta\vu_\delta, W, E_\delta(\varrho_\delta,\vu_\delta), \nu_\delta\right]\right\}
$$
is tight in a space
$$
\mathcal X^\delta = \mathcal X^\delta_{\varrho_0}\times\mathcal X_{\vq_0}\times \mathcal X_{\vq_0/\sqrt{\varrho_0}}\times \mathcal X^\delta_\varrho\times\mathcal X^\delta_{\varrho\vu}\times\mathcal X_\vu \times\mathcal X_W \times \mathcal X_E\times\mathcal X_\nu
$$
where most of the spaces where defined in the previous subsection and, apart of them, we define
\begin{equation*}
\begin{split}
\mathcal X^\delta_{\varrho_0} & = L^\gamma(\TT)\\
\mathcal X^\delta_{\varrho} & = L^{\gamma+\beta}(((0,T)\times \mathbb T^3),w)\cap C_w([0,T],L^\gamma(\TT))\\
\mathcal X^\delta_{\varrho\vu} & = C_w([0,T],L^{2\gamma/(\gamma+1)}(\TT))\cap C([0,T],W^{-k,2}(\TT)),\ k>3/2.
\end{split}
\end{equation*}

Similarly to previous subsection we deduce the tightness of the above-mentioned law and, therefore, we obtain the existence of 
$$
\left(\tilde\varrho_{0,\delta},\tilde\vq_{0,\delta}, \tilde{\frac{\vq_{0,\delta}}{\sqrt{\varrho_{0,\delta}}}}, \tilde\varrho_\delta,\widetilde{\varrho_\delta\vu_\delta}, \tilde W, \tilde E_\delta(\varrho_\delta,\vu_\delta), \nu_\delta\right)
$$
and
$$
\left(\tilde\varrho_{0},\tilde\vq_{0}, \tilde{\frac{\vq_{0}}{\sqrt{\varrho_{0}}}}, \tilde\varrho,\widetilde{\varrho\vu}, \tilde W, \tilde E(\varrho,\vu), \nu\right)
$$
which fulfill
\begin{enumerate}
\item the laws of 
$$
\left(\varrho_{0,\delta},\vq_{0,\delta}, \frac{\vq_{0,\delta}}{\sqrt{\varrho_{0,\delta}}}, \varrho_\delta,\varrho_\delta\vu_\delta, W, E_\delta(\varrho_\delta,\vu_\delta), \nu_\delta\right)
$$
and
$$
\left(\tilde\varrho_{0,\delta},\tilde\vq_{0,\delta}, \tilde{\frac{\vq_{0,\delta}}{\sqrt{\varrho_{0,\delta}}}}, \tilde\varrho_\delta,\widetilde{\varrho_\delta\vu_\delta}, \tilde W, \tilde E_\delta(\varrho_\delta,\vu_\delta), \nu_\delta\right)
$$
coincide,
\item the law of 
$$
\left(\tilde\varrho_{0},\tilde\vq_{0}, \tilde{\frac{\vq_{0}}{\sqrt{\varrho_{0}}}}, \tilde\varrho,\widetilde{\varrho\vu}, \tilde W, \tilde E(\varrho,\vu), \nu\right)
$$
is a Radon measure on $\mathcal X^\delta$,
\item 
$$
\left(\tilde\varrho_{0,\delta},\tilde\vq_{0,\delta}, \tilde{\frac{\vq_{0,\delta}}{\sqrt{\varrho_{0,\delta}}}}, \tilde\varrho_\delta,\widetilde{\varrho_\delta\vu_\delta}, \tilde W, \tilde E_\delta(\varrho_\delta,\vu_\delta), \nu_\delta\right)
$$
converge to 
$$
\left(\tilde\varrho_{0},\tilde\vq_{0}, \tilde{\frac{\vq_{0}}{\sqrt{\varrho_{0}}}}, \tilde\varrho,\widetilde{\varrho\vu}, \tilde W, \tilde E(\varrho,\vu), \nu\right)
$$
in the topology of $\mathcal X^\delta$ $\mathbb P$-almost surely.
\end{enumerate}
Due to the above-mentioned convergence, we get that the limit functions satisfy the continuity equation \eqref{main.sys}$_1$ and the momentum equation in the form
\begin{multline*}
\int_0^T\int_\TT \tilde\varrho \tilde\vu \cdot \pat\varphi \ {\rm d}x{\rm d}t  + \int_\TT \tilde\varrho_0\tilde\vu_0 \varphi(0,\cdot)\ {\rm d}x -\int_\TT \tilde\varrho(T,\cdot)\tilde\vu(T,\cdot) \varphi(T,\cdot)\ {\rm d}x\\
+ \int_0^T \int_\TT \tilde\varrho \tilde\vu\otimes\tilde\vu:\nabla \varphi + \overline{p\delta(\varrho)}\diver \varphi\ {\rm d}x{\rm d}t
 - \int_0^T\int_\TT S(D\tilde\vu): D\varphi  \ {\rm d}x{\rm d}t = \\
 +
\int_0^T\int_\TT \tilde\varrho \nabla K*\tilde\varrho  \varphi \ {\rm d}x{\rm d}t\\
 - \int_0^T\int_\TT\tilde\varrho(t,x) \varphi(x) \int_\TT \tilde\varrho(t,y) \psi(x-y) (\tilde\vu(t,y) - \tilde\vu(t,x)) \ {\rm d}x{\rm d}y{\rm d}t\\
+ \sum_{k=1}^\infty\int_0^T\int_\TT \overline{\vG_k(\varrho,\varrho\vu)}\cdot\varphi \ {\rm d}x{\rm d}W_k
\end{multline*}
where $\overline{p(\varrho)}$ is a weak limit of $p_\delta(\tilde\varrho_\delta)$ and, similarly, $\overline{\vG_k(\varrho,\varrho\vu)}$ is a weak limit of $\vG_k(\tilde\varrho_\delta,\tilde\varrho_\delta\tilde\vu_\delta)$ (here we note that the convergence of $\tilde\varrho_\delta\tilde\vu_\delta\otimes\tilde\vu_\delta$ to its counterpart can be obtained similarly as in the previous subsection).

 It remains to show that
$$
\overline{p(\varrho)} = p(\tilde\varrho),\ \mbox{ and }\overline{\vG_k(\varrho,\varrho\vu)} = \vG_k(\tilde\varrho,\tilde\varrho\tilde\vu).
$$
In order to reach this goal it is sufficient to show that $$\tilde\varrho_\delta\to \tilde\varrho$$ strongly in $L^1((0,T),L^1(\TT))$ $\mathbb P$ almost surely.

Let $k\in \mathbb N$. We define
\begin{equation*}
T_k(z) = kT\left(\frac zk\right),\ z\in \mathbb R
\end{equation*}
where $T$ is a smooth concave function satisfying
\begin{equation*}
T(z) = \left\{\begin{array}{l} z\ \mbox{for}\ z\leq 1\\ 2\ \mbox{for}\ z\geq 3\end{array}\right.	
\end{equation*}
We use $\Delta^{-1}\left(T_k(\tilde\varrho_\delta) - \left(T_k(\tilde\varrho_\delta)\right)_\TT\right)$ as a test function in the momentum equation in order to obtain
\begin{equation}\label{evf.for.delta}
\lim_{\delta\to0}\int_0^T\int_\TT \left(p_\delta(\tilde\varrho) - (\lambda + 2\mu)\diver \tilde \vu_\delta\right) T_k(\tilde \varrho_\delta)\ {\rm d}x{\rm d}t = \int_0^T\int_\TT \left(\overline{p(\tilde\varrho)} - (\lambda + 2\mu)\diver \tilde \vu\right) T_k(\tilde\varrho)\ {\rm d}x{\rm d}t.
\end{equation}
We do not provide a proof of this equality as it can be proven similarly to \eqref{eq:evf}.
Assume for a while that $\tilde\varrho_\delta, \ \tilde\vu_\delta$ as well as $\tilde\varrho$ and $\tilde \vu$ satisfies the renormalized continuity equation, i.e.,
\begin{equation}\begin{split}\label{renorm.eq.delta}
\pat b(\tilde\varrho_\delta) + \diver (b(\tilde\varrho_\delta) \tilde\vu_\delta) + (b'(\tilde\varrho_\delta)\tilde\varrho_\delta - b(\tilde\varrho_\delta))\diver\tilde\vu_\delta & = 0,\\
\pat b(\tilde\varrho) + \diver (b(\tilde\varrho) \tilde\vu) + (b'(\tilde\varrho)\tilde\varrho - b(\tilde\varrho))\diver\tilde\vu & = 0.\\
\end{split}\end{equation}
We take $b = L_k$ where $L_k$ is defined as
\begin{equation*}
L_k(z) = \left\{\begin{array}{l} z\log z\ \mbox{for }0\leq z\leq k,\\ z\log k + z\int_k^z \frac 1{s^2} T_k(s)\ {\rm d}s.\end{array}\right.
\end{equation*}
We deduce
\begin{equation*}
\pat L_k(\tilde\varrho_\delta) + \diver (L_k(\tilde\varrho_\delta)\tilde\vu_\delta) + T_k(\tilde\varrho_\delta)\diver \tilde\vu_\delta = 0
\end{equation*}
and
\begin{equation*}
\pat L_k(\tilde\varrho) + \diver (L_k(\tilde\varrho)\tilde\vu) + T_k(\tilde\varrho)\diver \tilde\vu = 0
\end{equation*}
and similarly to \eqref{eq:outcome.of.evf} we deduce
\begin{equation*}
\int_\TT L_k(\tilde\varrho(t,\cdot)) - \overline{L_k(\tilde\varrho(t,\cdot))}\ {\rm d}x \geq0
\end{equation*}
for almost all $t\in [0,T]$. We send $k\to\infty$ to deduce that $\tilde\varrho\log\tilde\varrho = \overline{\tilde\varrho\log\tilde\varrho}$ which yields $\tilde\varrho_\delta\to \tilde\varrho$ strongly in $L^1((0,T),L^1(\TT))$ almost surely in $\mathbb P$. 

It remains to prove that the renormalized equation \eqref{renorm.eq.delta}$_2$ holds true. 
Recall that $T_k$ is increasing and thus
\begin{multline*}
\int_0^T\int_\TT \tilde\varrho_\delta^\gamma T_k(\tilde\varrho_\delta) - \overline{\tilde\varrho^\gamma}\overline{T_k(\tilde\varrho)} \ {\rm d}x{\rm d}t \geq \int_0^T\int_\TT (\tilde\varrho_\delta^\gamma - \tilde\varrho^\gamma)(T_k(\tilde\varrho_\delta^\gamma) - T_k(\tilde\varrho))\ {\rm d}x{\rm d}t \\
\geq \int_0^T\int_\TT |T_k(\tilde\varrho_\delta) - T_k(\tilde\varrho)|^{\gamma+1}\ {\rm d}x{\rm d}t.
\end{multline*}
We deduce by help of \eqref{evf.for.delta} that
\begin{multline*}
0 = \lim_{\delta\to0} \int_0^T\int_\TT \tilde\varrho_\delta^\gamma T_k(\tilde\varrho_\delta) - \overline{\tilde\varrho^\gamma}\overline{T_k(\tilde\varrho)} - (\lambda + 2\mu)\left(\diver \tilde\vu_\delta T_k(\tilde\varrho_\delta) - \diver\tilde\vu \overline{T_k(\tilde\varrho)}\right)\ {\rm d}x{\rm d}t\\
\geq \|T_k(\tilde\varrho_\delta) - T_k(\tilde\varrho) \|_{L^{\gamma+1}((0,T)\times\TT)}^{\gamma+1} - 2(\lambda + 2\mu) \|\diver \tilde\vu_\delta\|_{L^2((0,T)\times\TT)}\|T_k(\tilde\varrho_\delta) - T_k(\tilde\varrho)\|_{L^2((0,T)\times\TT)}
\end{multline*}
and we get
\begin{equation*}
\mathbb E \left(\|T_k(\tilde\varrho_\delta) - T_k(\tilde\varrho)\|^{\gamma+1}_{L^{\gamma+1}((0,T)\times\TT)}\right) \leq c
\end{equation*}
with $c$ independent of $k$ and $\delta$. 

This is enough to deduce that
$$
\mathbb E \left(\int_0^T\int_\TT \overline{b'(T_k(\tilde\varrho))} \overline{\left(T'_k(\tilde\varrho)\tilde\varrho - T_k(\tilde\varrho)\right)\diver \tilde\vu} \ {\rm d}x{\rm d}t\right)\to 0
$$
and the validity of \eqref{renorm.eq.delta} follows -- compare this with the proof of \cite[Lemma 3.8]{FeNo}.\bigskip\\

\noindent {\bf Acknowledgment}: The work of V\'aclav M\'acha was supported by the Czech Science Foundation (GA\v CR), Grant Agreement GA18-05974S in the framework of RVO:67985840. The work of Pavel Ludv\'ik was supported by the Grant IGA\_PrF\_2021\_008 ``Mathematical Models'' of the Internal Grant Agency of Palack\'y University in Olomouc.

\bibliographystyle{plain}
\bibliography{literatura}

\begin{thebibliography}{10}

\bibitem{BFH17}
Dominic Breit, Eduard Feireisl, and Martina Hofmanov\'{a}.
\newblock Compressible fluids driven by stochastic forcing: the relative energy
  inequality and applications.
\newblock {\em Comm. Math. Phys.}, 350(2):443--473, 2017.

\bibitem{BrFeHo}
Dominic Breit, Eduard Feireisl, and Martina Hofmanov\'{a}.
\newblock {\em Stochastically forced compressible fluid flows}, volume~3 of
  {\em De Gruyter Series in Applied and Numerical Mathematics}.
\newblock De Gruyter, Berlin, 2018.

\bibitem{BrOnSe16}
Zdzis{\l}aw Brze\'{z}niak, Martin Ondrej\'{a}t, and Jan Seidler.
\newblock Invariant measures for stochastic nonlinear beam and wave equations.
\newblock {\em J. Differential Equations}, 260(5):4157--4179, 2016.

\bibitem{BrMa}
Jan B\v{r}ezina and V\'{a}clav M\'{a}cha.
\newblock Inviscid limit for the compressible {E}uler system with non-local
  interactions.
\newblock {\em J. Differential Equations}, 267(7):4410--4428, 2019.

\bibitem{CaCaRo}
J.~A. Ca\~{n}izo, J.~A. Carrillo, and J.~Rosado.
\newblock A well-posedness theory in measures for some kinetic models of
  collective motion.
\newblock {\em Math. Models Methods Appl. Sci.}, 21(3):515--539, 2011.

\bibitem{CaFeGwSw}
Jos\'{e}~A. Carrillo, Eduard Feireisl, Piotr Gwiazda, and Agnieszka
  \'{S}wierczewska Gwiazda.
\newblock Weak solutions for {E}uler systems with non-local interactions.
\newblock {\em J. Lond. Math. Soc. (2)}, 95(3):705--724, 2017.

\bibitem{ChKrMaSc}
Elisabetta Chiodaroli, Ond\v{r}ej Kreml, V\'{a}clav M\'{a}cha, and Sebastian
  Schwarzacher.
\newblock Non--uniqueness of admissible weak solutions to the compressible
  {E}uler equations with smooth initial data.
\newblock {\em Trans. Amer. Math. Soc.}, 374(4):2269--2295, 2021.

\bibitem{CuSm}
Felipe Cucker and Steve Smale.
\newblock Emergent behavior in flocks.
\newblock {\em IEEE Trans. Automat. Control}, 52(5):852--862, 2007.

\bibitem{DPZ92}
Giuseppe Da~Prato and Jerzy Zabczyk.
\newblock {\em Stochastic equations in infinite dimensions}, volume~44 of {\em
  Encyclopedia of Mathematics and its Applications}.
\newblock Cambridge University Press, Cambridge, 1992.

\bibitem{DeSz}
Camillo De~Lellis and L\'{a}szl\'{o} Sz\'{e}kelyhidi, Jr.
\newblock The {E}uler equations as a differential inclusion.
\newblock {\em Ann. of Math. (2)}, 170(3):1417--1436, 2009.

\bibitem{DGHT11}
Arnaud Debussche, Nathan Glatt-Holtz, and Roger Temam.
\newblock Local martingale and pathwise solutions for an abstract fluids model.
\newblock {\em Phys. D}, 240(14-15):1123--1144, 2011.

\bibitem{DeHiPr}
Robert Denk, Matthias Hieber, and Jan Pr\"{u}ss.
\newblock R-boundedness, {F}ourier multipliers and problems of elliptic and
  parabolic type.
\newblock {\em Mem. Amer. Math. Soc.}, 166(788):viii+114, 2003.

\bibitem{DiRuSc}
Lars Diening, Michael R\r{u}\v{z}i\v{c}ka, and Katrin Schumacher.
\newblock A decomposition technique for {J}ohn domains.
\newblock {\em Ann. Acad. Sci. Fenn. Math.}, 35(1):87--114, 2010.

\bibitem{feireisl}
Eduard Feireisl.
\newblock {\em Dynamics of viscous compressible fluids}, volume~26 of {\em
  Oxford Lecture Series in Mathematics and its Applications}.
\newblock Oxford University Press, Oxford, 2004.

\bibitem{FeNo}
Eduard Feireisl and Anton\'{\i}n Novotn\'{y}.
\newblock {\em Singular limits in thermodynamics of viscous fluids}.
\newblock Advances in Mathematical Fluid Mechanics. Birkh\"{a}user Verlag,
  Basel, 2009.

\bibitem{FeNoPe}
Eduard Feireisl, Anton\'{\i}n Novotn\'{y}, and Hana Petzeltov\'{a}.
\newblock On the existence of globally defined weak solutions to the
  {N}avier-{S}tokes equations.
\newblock {\em J. Math. Fluid Mech.}, 3(4):358--392, 2001.

\bibitem{JAK97}
Adam Jakubowski.
\newblock The almost sure skorokhod representation for subsequences in
  nonmetric spaces.
\newblock {\em Teoriya Veroyatnostei i ee Primeneniya}, 42:209--216, 1997.

\bibitem{KS91}
Ioannis Karatzas and Steven~E. Shreve.
\newblock {\em Brownian motion and stochastic calculus}, volume 113 of {\em
  Graduate Texts in Mathematics}.
\newblock Springer-Verlag, New York, second edition, 1991.

\bibitem{KaMeTr}
Trygve~K. Karper, Antoine Mellet, and Konstantina Trivisa.
\newblock Existence of weak solutions to kinetic flocking models.
\newblock {\em SIAM J. Math. Anal.}, 45(1):215--243, 2013.

\bibitem{lunardi}
Alessandra Lunardi.
\newblock {\em Analytic semigroups and optimal regularity in parabolic
  problems}.
\newblock Modern Birkh\"{a}user Classics. Birkh\"{a}user/Springer Basel AG,
  Basel, 1995.
\newblock [2013 reprint of the 1995 original] [MR1329547].

\bibitem{MR16}
Carlo Marinelli and Michael R\"{o}ckner.
\newblock On the maximal inequalities of {B}urkholder, {D}avis and {G}undy.
\newblock {\em Expo. Math.}, 34(1):1--26, 2016.

\bibitem{NoSt}
A.~Novotn\'{y} and I.~Stra\v{s}kraba.
\newblock {\em Introduction to the mathematical theory of compressible flow},
  volume~27 of {\em Oxford Lecture Series in Mathematics and its Applications}.
\newblock Oxford University Press, Oxford, 2004.

\end{thebibliography}

\end{document}